\documentclass[12pt,a4paper, oneside,reqno]{article}
\usepackage[utf8]{inputenc}
\usepackage[english]{babel}
\usepackage[left=1in,right=1in,top=1in,bottom=1in]{geometry}
\usepackage{amsfonts,amssymb,amsmath,amsthm}

\usepackage[symbol]{footmisc}
\usepackage{longtable}
\usepackage{tasks}
\usepackage{enumerate}
\usepackage{cite}
\usepackage{xcolor}
\usepackage{hyperref}
\usepackage{verbatim}
\usepackage{tkz-graph}
\usepackage{fancyhdr}
\theoremstyle{plain}
\newtheorem{theorem}{Theorem}
\newtheorem{lemma}[theorem]{Lemma}
\newtheorem{proposition}[theorem]{Proposition}
\newtheorem{remark}[theorem]{Remark}

\newtheorem{definition}[theorem]{Definition}
\newtheorem{corollary}[theorem]{Corollary}
\usetikzlibrary{arrows}

 \usepackage{fouriernc}

\newcommand{\Der}{\mathfrak{Der}}
\newcommand{\QDer}{{\rm Q}\mathfrak{Der}}

\newcommand{\Inn}{\operatorname{Inn}}
\newcommand{\ad}{\operatorname{ad}}

\newcommand{\C}{\mathbb{C}}
\newcommand{\Z}{\mathbb{Z}}
\newcommand{\SV}{\mathfrak{sv}}
\newcommand{\ESV}{\mathfrak{esv}}
\newcommand{\Id}{\operatorname{Id}}
\newcommand{\Hom}{\operatorname{Hom}}
\newcommand{\Span}[1]{\langle #1\rangle}
\begin{document}

\noindent{\large
Quasi-derivations of the (extended) Schr\"odinger--Witt algebras}\footnote{The work is supported by FCT 2023.08031.CEECIND and UID/00212/2025.} \footnote{The original concept, theoretical framework, and primary results were developed by I.K., A.K., and Z.S.;
while Aristotle  (https://aristotle.harmonic.fun/)  acted as a non-human research assistant, assisting in checking and simplifying proofs.}

\begin{center}
	{\bf
		Ivan Kaygorodov\footnote{CMA-UBI, University of Beira Interior, Portugal; \ kaygorodov.ivan@gmail.com},
   Abror Khudoyberdiyev\footnote{V.I.Romanovskiy Institute of Mathematics, Uzbekistan Academy of Sciences; National University of Uzbekistan, Uzbekistan; \
khabror@mail.ru} \&
Zarina Shermatova\footnote{V.I.Romanovskiy Institute of Mathematics, Uzbekistan Academy of Sciences; Kimyo International University in Tashkent, Uzbekistan; \
zarina\_shermatova91@mail.ru}}
\end{center}


\noindent 
{\bf Abstract:}
{\it The Schr\"odinger--Witt algebras are infinite-dimensional Lie algebras of local
dynamical symmetries of the free Schr\"odinger equation; they contain the Witt algebra of
conformal transformations of the time line and appear in non-equilibrium statistical
physics and in non-relativistic conformal field theory.
We describe all quasi-derivations of the original and the twisted
Schr\"odinger--Witt\footnote{We use the name \textquotedblleft Witt\textquotedblright\ for the simple Witt algebra and the name \textquotedblleft Virasoro\textquotedblright\ for its universal central extension; in the physical literature the algebras studied here are usually called Schr\"odinger--Virasoro algebras.} algebras. We prove that every quasi-derivation is the sum of
a derivation, a trivial $\frac12$-derivation, a linear map with values in the center, and a
finite linear combination of maps from three new families of quasi-derivations.
As corollaries, we describe all $\delta$-derivations, all $\delta$-biderivations and all transposed $\delta$-Poisson structures of the
Schr\"odinger--Witt algebras. Finally, we prove that every quasi-derivation of the (original or twisted) extended
Schr\"odinger--Witt algebra is the sum of an inner derivation and a multiple of the
identity map. As a consequence, there are no
nontrivial $\delta$-derivations (for $\delta\notin \{\frac12,1\}$), $\delta$-biderivations (for $\delta\neq1$), and
transposed
$\delta$-Poisson structures on these algebras.}

\medskip 



\tableofcontents

\newpage

\section*{Introduction}

\textbf{Schr\"odinger symmetry in physics.}
The maximal group of kinematic symmetries of the free Schr\"odinger equation
$2i\partial_t\psi=-\Delta\psi$ was found by Niederer \cite{Nie72} and Hagen \cite{Hag72}. Its
Lie algebra, the Schr\"odinger algebra, contains the Galilei algebra, the dilatations
$t\mapsto\lambda^2t$, $x\mapsto\lambda x$ (with dynamical exponent $z=2$) and the special
Schr\"odinger transformations $t\mapsto t/(1+\kappa t)$. In order to describe strongly
anisotropic critical systems, Henkel \cite{H94} embedded the Schr\"odinger algebra (in one
space dimension) into an infinite-dimensional Lie algebra, the Schr\"odinger--Virasoro
algebra, which is the main object of the present paper (we call it the Schr\"odinger--Witt
algebra, see below). In this algebra, the generators $L_m$ realize the time
reparametrizations $t\mapsto f(t)$, i.e., a copy of the Witt algebra, while $Y_n$ generate
time-dependent space translations (including the space translations and the Galilei boosts)
and $M_m$ generate time-dependent phase shifts of the wave function ($M_0$ corresponds to the
mass, which is central). The local scale invariance generated by this algebra has been
applied to dynamical scaling and ageing phenomena in systems far from equilibrium
\cite{Hen02,HU03,HP10}. Non-relativistic conformal field theories with Schr\"odinger symmetry
also describe cold atoms at unitarity \cite{NS07}, and their holographic duals \cite{Son08,BM08}
have made Schr\"odinger symmetry a part of the non-relativistic AdS/CFT correspondence.
A systematic mathematical study of the Schr\"odinger--Virasoro group and algebra, including
their representation theory and cohomology, was carried out by Roger and Unterberger
\cite{RU06}; see also the monograph \cite{UR12}. Studying vertex algebra representations of
the Schr\"odinger--Virasoro algebra, Unterberger \cite{Unt09} introduced the extended
Schr\"odinger--Virasoro algebra, which contains one more family of generators $N_m$
(denoted by $\ESV(\frac12)$ in this paper).

\medskip
\noindent 
\textbf{Algebras with the same multiplication table.}
Lie algebras with the multiplication table \eqref{eq:sv} (or \eqref{eq:esv}) have been
studied from many algebraic points of view: representations \cite{LS08,ZTL10,UR12},
automorphisms and Verma modules \cite{TZ09}, derivations and cohomology \cite{RU06},
structure, derivations and automorphisms of the extended algebra \cite{GJP09}, biderivations
and commutative post-Lie algebra structures \cite{LT19}, $\frac12$-derivations and transposed
Poisson structures \cite{YH22,Sh24}. Since these algebras are centerless and the Virasoro
algebra is the central extension of the Witt algebra, we follow \cite{ZTL10} and use the
name Schr\"odinger--Witt algebras; the name Schr\"odinger--Virasoro is used for the same
algebras in \cite{H94,RU06,UR12,Unt09,LS08,TZ09,GJP09,LT19,YH22,Sh24}.

\medskip
\noindent 
\textbf{Generalized derivations.} Derivations and their generalizations are among the most important tools for studying algebraic structures.
Among the many generalizations of derivations we mention $\delta$-derivations, introduced by Filippov \cite{Fil98} and studied for many classes of
algebras and superalgebras \cite{Kay07,Kay09,zz,Zus10}, and generalized derivations, quasi-derivations and the quasicentroid, introduced and systematically studied by Leger and Luks \cite{LL00}.
Recall that a linear map $D$ of a Lie algebra $L$ is a
\emph{quasi-derivation} if there is a linear map $D'$ of $L$ such that
\[
[D(x),y]+[x,D(y)]=D'([x,y])\quad \text{for all } x,y\in L.
\]
Every derivation is a quasi-derivation (with $D'=D$) and every $\delta$-derivation with
$\delta\ne0$ is a quasi-derivation (with $D'=\delta^{-1}D$). The $\frac 12$-derivations have
attracted a lot of attention in recent years, since Ferreira, Kaygorodov and Lopatkin established a relation between $\frac12$-derivations of Lie algebras and transposed Poisson algebras \cite{FKL21}. The notion of transposed Poisson algebras was introduced by Bai, Bai, Guo and Wu \cite{BBGW20} as a dual notion of Poisson algebras; transposed Poisson algebras are closely related to Novikov--Poisson algebras and $3$-Lie algebras \cite{BBGW20}, and Novikov algebras originate from Poisson brackets of hydrodynamic type \cite{BN85}. Their seminal paper attracted considerable interest, and since then transposed Poisson algebras have been studied in various directions (see \cite{kms,deltu,delta,K24,aae23,LS26,KLZ23} and references therein).
In particular,
transposed Poisson structures have been described on many important Lie algebras:
the Witt and Virasoro algebras \cite{FKL21}; Galilean algebras \cite{KLZ23}; twisted Heisenberg--Virasoro, Schr\"odinger--Witt, and extended Schr\"odinger--Witt algebras \cite{YH22}; the original deformative Schr\"odinger--Witt algebras and the central extension of the extended Schr\"odinger--Witt algebras \cite{Sh24}; and not-finitely graded Witt-type algebras \cite{KKS25}.
Biderivations and the closely related commuting maps of Lie algebras were studied  in \cite{HWX16,LT19}, and cohomological methods for Virasoro-type structures
were developed, for example, in \cite{RU06,akl}.
Some of the cited results were generalized to
quasi-derivations and to the description of transposed $\delta$-Poisson structures:
for the Witt algebra and the $\mathcal{W}(a,b)$ algebras    \cite{kkz}, and 
for the Block Lie algebras $\mathcal{B}(q)$ \cite{kkz2}.
The present paper extends the known descriptions of $\frac 12$-derivations and transposed Poisson structures of the (extended) Schr\"odinger--Witt algebras to descriptions of their quasi-derivations,   $\delta$-(bi)derivations and transposed $\delta$-Poisson structures.

 \medskip
\noindent 
\textbf{Main results.} The derivations of both the original and twisted Schr\"odinger--Witt algebras were determined in \cite{RU06}. Further investigations in \cite{LT19} characterized their biderivations and commutative post-Lie algebra structures. Moreover, Yuan and Hua \cite{YH22} established that every $\frac12$-derivation of these algebras is trivial, namely, a scalar multiple of the identity map.
In Subsection \ref{section4} we describe all quasi-derivations of the Schr\"odinger--Witt
algebras $\SV(\varepsilon)$ (Theorem \ref{thm:main}): every quasi-derivation is the sum of
a derivation, a multiple of the identity map, a linear map with values in the center and a
finite linear combination of maps from three new families $\Omega_t$, $\Psi_t$ $(t\neq0)$ and
$\Theta_s$ of quasi-derivations, which are neither derivations nor
$\delta$-derivations. As corollaries, we describe all   $\delta$-(bi)derivations and all transposed $\delta$-Poisson structures on
$\SV(\varepsilon)$ for every $\delta\in\C$ (Corollaries \ref{cor:delta}, \ref{cor:bider} and
\ref{cor:tp}); in particular, we obtain new proofs of the descriptions of derivations
\cite{RU06} and $\frac12$-derivations \cite{YH22} of $\SV(\varepsilon)$.
In the last section we consider the extended Schr\"odinger--Witt algebras
$\ESV(\varepsilon)$. Here the
situation is much more rigid: every quasi-derivation of $\ESV(\varepsilon)$ is of the form
$\ad_x+\lambda\Id$ (Theorem \ref{thm:esvmain}). Consequently, all derivations of
$\ESV(\varepsilon)$ are inner (Corollary \ref{cor:esvder}; for the original extended
Schr\"odinger--Witt algebra $\ESV(\frac12)$ this was proved earlier by Gao, Jiang and
Pei \cite{GJP09}), its $\delta$-derivations are described for all $\delta\in\C$
(Corollary \ref{cor:esvdelta}), every biderivation of $\ESV(\varepsilon)$ is a multiple of the
Lie bracket and every $\delta$-biderivation with $\delta\neq1$ is zero (Corollary
\ref{thm:esvbider}), and all transposed $\delta$-Poisson structures on $\ESV(\varepsilon)$ are
trivial for every $\delta\in\C$ (Corollary \ref{cor:esvtp}). In particular,  all transposed Poisson structures on
$\ESV(\varepsilon)$ are trivial; for the original extended Schr\"odinger--Witt algebra
$\ESV(\frac12)$ these results were obtained earlier in \cite{YH22}, and our approach gives
a new proof of them.

\medskip
\noindent{\bf Perspectives.}
From the point of view of mathematical physics, these results describe the
\textquotedblleft generalized infinitesimal symmetries\textquotedblright\ of the
Schr\"odinger--Witt symmetry algebras and show their rigidity: neither $\SV(\varepsilon)$ nor
$\ESV(\varepsilon)$ admits a nontrivial compatible commutative associative product of
transposed Poisson type (for $\SV(\varepsilon)$ and $\delta=1$ the only such products are
supported on the abelian ideal $\Span{M_m\mid m\in\Z}$ of phase shifts), and the extended algebra admits
no outer derivations at all, so that its symmetry structure cannot be deformed or enlarged
within these classes of generalized derivations. In contrast, the quasi-derivations
$\Omega_t,\Psi_t,\Theta_s$ show that $\SV(\varepsilon)$ carries generalized symmetries that
do not come from derivations of its Witt subalgebra.

\section{Preliminaries}\label{section1}

All algebras and vector spaces are considered over the complex field $\C$. The notation
$\Span{S}$ means the $\C$-subspace generated by $S$. For a Lie algebra $L$ and $x\in L$ we
denote by $\ad_x$ the inner derivation $\ad_x(y)=[x,y]$. The center of $L$ is denoted by
$Z(L)$.

\begin{definition}[see \cite{LL00}]\label{def:qd}
Let $(L,[-,-])$ be a Lie algebra and $D:L\to L$ be a linear map. Then $D$ is a
\emph{quasi-derivation} if there is a linear map $D':L\to L$ such that
\begin{equation}\label{eq:qd}
[D(x),y]+[x,D(y)]=D'([x,y]) \quad\text{for all } x,y\in L.
\end{equation}
The space of all quasi-derivations of $L$ is denoted by $\QDer(L)$. If $D$ and $D'$
satisfy \eqref{eq:qd}, we say that $D'$ is \emph{associated} with $D$.
\end{definition}

\begin{definition}[see \cite{Fil98}]\label{def:dd}
Let $(L,[-,-])$ be a Lie algebra, $\delta\in\C$ and $\varphi:L\to L$ be a linear map. Then
$\varphi$ is a \emph{$\delta$-derivation} if
\begin{equation}\label{eq:dd}
\varphi([x,y])=\delta\big([\varphi(x),y]+[x,\varphi(y)]\big)\quad\text{for all } x,y\in L.
\end{equation}
The space of $\delta$-derivations of $L$ is denoted by $\Delta_\delta(L)$. In particular,
$\Delta_1(L)=\Der \ L$.
\end{definition}

 If $\delta = \frac{1}{2}$, then $\varphi$ is called a \emph{$\frac{1}{2}$-derivation}. Multiplications by elements of the ground field are called \emph{trivial} $\frac{1}{2}$-derivations.

\begin{definition}[see \cite{YH22}]\label{def:bider}
Let $(L,[-,-])$ be a Lie algebra and $\delta\in\C$. A bilinear map $\phi:L\times L\to L$ is a
\emph{$\delta$-biderivation} if
\begin{align}
\phi([x,y],z)&=\delta\big([\phi(x,z),y]+[x,\phi(y,z)]\big), \label{eq:bd1}\\
\phi(x,[y,z])&=\delta\big([\phi(x,y),z]+[y,\phi(x,z)]\big) \label{eq:bd2}
\end{align}
for all $x,y,z\in L$. The $1$-biderivations are called \emph{biderivations}.
\end{definition}

In other words, $\phi$ is a $\delta$-biderivation if and only if for every $z\in L$ the maps
$x\mapsto\phi(x,z)$ and $x\mapsto\phi(z,x)$ are $\delta$-derivations. For every
$\lambda\in\C$ the bilinear map $\phi(x,y)=\lambda[x,y]$ is a biderivation, by the Jacobi
identity. Biderivations and the closely related commuting maps have been studied
extensively, for associative rings as well as for Lie algebras (see, for example,
\cite{HWX16}).

\begin{definition}\label{def:tp}
Let $(L,[-,-])$ be a Lie algebra and $\delta\in\C$. A \emph{transposed $\delta$-Poisson
structure} on $L$ is a commutative associative multiplication $\cdot$ on $L$ such that
\begin{equation}\label{eq:tdp}
z\cdot[x,y]=\delta\big([z\cdot x,y]+[x,z\cdot y]\big)\quad\text{for all } x,y,z\in L.
\end{equation}
The transposed $\frac12$-Poisson structures, i.e., the commutative associative
multiplications $\cdot$ on $L$ such that $2z\cdot[x,y]=[z\cdot x,y]+[x,z\cdot y]$ for all
$x,y,z\in L$, are called \emph{transposed Poisson structures} \cite{BBGW20,FKL21}.
A transposed $\delta$-Poisson structure is \emph{trivial} if $x\cdot y=0$ for all $x,y\in L$.
\end{definition}

In other words, \eqref{eq:tdp} means that all multiplication operators $x\mapsto z\cdot x$ are
$\delta$-derivations of the Lie algebra $(L,[-,-])$.

\begin{lemma}[see \cite{YH22}]\label{lem:tpbider}
Let $\cdot$ be a transposed $\delta$-Poisson structure on a Lie algebra $L$. Then the bilinear
map $\phi(x,y)=x\cdot y$ is a $\delta$-biderivation of $L$. In particular, if $\cdot$ is a
transposed Poisson structure on $L$, then $\phi$ is a $\frac12$-biderivation of $L$.
\end{lemma}

\begin{lemma}\label{lem:unique}
Let $L$ be a perfect Lie algebra, i.e., $[L,L]=L$. Then for every $D\in\QDer(L)$ the map $D'$
associated with $D$ is unique. In particular, the map $D\mapsto D'$ is linear, $D'=0$
whenever $D$ is a central map (i.e., a linear map with values in $Z(L)$), and $D\in\QDer(L)$ is a derivation if and only if $D'=D$.
\end{lemma}

\begin{proof}
If $D_1'$ and $D_2'$ are associated with $D$, then $(D_1'-D_2')([x,y])=0$ for all
$x,y\in L$, and hence $D_1'=D_2'$ on $[L,L]=L$. The remaining statements are straightforward; note that every central map $D$ is a quasi-derivation with associated map $0$, since $[D(x),y]+[x,D(y)]=0$.
\end{proof}

Let $G$ be an abelian group and $L=\bigoplus_{g\in G}L_g$ be a $G$-graded Lie algebra. For
$v\in L$ and $k\in G$ we denote by $v_k$ the homogeneous component of degree $k$ of $v$. For
a linear map $D:L\to L$ and $g\in G,$ we define the linear map $D_g$ by
\[
D_g(x)=\big(D(x)\big)_{h+g}\quad\text{for } x\in L_h.
\]
We call $D_g$ the \emph{homogeneous component of degree $g$} of $D$. Clearly,
$D_g(L_h)\subseteq L_{h+g}$ and, for every $x\in L$, we have $D(x)=\sum\limits_{g\in G}D_g(x)$,
where only finitely many summands are nonzero. In this sense we write $D=\sum\limits_{g\in G}D_g$.
A linear map $D$ is called \emph{homogeneous of degree $g$} if $D=D_g$.

\begin{lemma}\label{lem:grading}
Let $L=\bigoplus_{g\in G}L_g$ be a $G$-graded Lie algebra and let $D\in\QDer(L)$ with
associated map $D'$. Then for every $g\in G$ the homogeneous component $D_g$ is a
quasi-derivation of $L$ and $D'_g$ is associated with $D_g$. The same holds for
$\delta$-derivations: every homogeneous component of a $\delta$-derivation is a
$\delta$-derivation.
\end{lemma}

\section{The Schr\"odinger--Witt algebras}

\begin{definition}[see \cite{H94}]\label{def:sv}
Let $\varepsilon\in\{0,\frac12\}$. The \emph{Schr\"odinger--Witt algebra}
$\SV(\varepsilon)$ is the Lie algebra with the basis
$\{L_m,\ Y_n,\ M_m \mid m\in\Z,\ n\in\varepsilon+\Z\}$ and the following nontrivial
relations:
\begin{equation}\label{eq:sv}
\begin{aligned}
&[L_m,L_n]=(m-n)L_{m+n}, && [L_m,Y_n]=\Big(\frac m2-n\Big)Y_{m+n},\\
&[L_m,M_n]=-nM_{m+n}, && [Y_m,Y_n]=(m-n)M_{m+n}.
\end{aligned}
\end{equation}
All other brackets of basis elements are zero. The algebra $\SV(\frac12)$ is called the
\emph{original} Schr\"odinger--Witt algebra and $\SV(0)$ is called the \emph{twisted}
Schr\"odinger--Witt algebra.
\end{definition}

The algebra $\SV(\varepsilon)$ is
graded by the abelian group $G=\frac12\Z\times\Z_2$:
\[
\SV(\varepsilon)_{(t,\bar0)}=\Span{L_t,M_t}\ (t\in\Z),\qquad
\SV(\varepsilon)_{(t,\bar1)}=\Span{Y_t}\ (t\in\varepsilon+\Z),
\]
and all other homogeneous components are zero; by \eqref{eq:sv},
\begin{center}
$[\SV(\varepsilon)_{(a,x)},\SV(\varepsilon)_{(b,y)}]\subseteq\SV(\varepsilon)_{(a+b,x+y)}$.
\end{center}Let $D$ be a linear map of $\SV(\varepsilon)$ and $(t,x)\in G$. If $x=\bar0$ and $t\notin\Z$,
or $x=\bar1$ and $t\notin\varepsilon+\Z$, then $D_{(t,x)}=0$, because the corresponding
homogeneous components of the image are zero. We say that a linear map is \emph{even of
degree $j\in\Z$} if it is homogeneous of degree $(j,\bar0)$, and \emph{odd of degree
$s\in\varepsilon+\Z$} if it is homogeneous of degree $(s,\bar1)$. Thus, every quasi-derivation of $\SV(\varepsilon)$ is the sum of its even
homogeneous components of degrees $j\in\Z$ and its odd homogeneous components of degrees
$s\in\varepsilon+\Z$, and all these components are quasi-derivations (Lemma \ref{lem:grading}).

\begin{lemma}[see \cite{RU06}]\label{lem:perfect}
The algebra $\SV(\varepsilon)$ is perfect and $Z(\SV(\varepsilon))=\Span{M_0}$.
\end{lemma}

By Lemmas \ref{lem:unique} and \ref{lem:perfect}, the map $D'$ associated with a
quasi-derivation $D$ of $\SV(\varepsilon)$ is unique. Moreover, if $D$ is a
homogeneous quasi-derivation of $\SV(\varepsilon)$, then $D'$ is homogeneous of the same degree, since by Lemma \ref{lem:grading} the homogeneous component of $D'$ of the degree of $D$ is also associated with $D$.

The description of derivations of $\SV(\varepsilon)$ is given in the following theorem.

\begin{theorem}[see \cite{RU06}]\label{thm:der}
$\Der\ \SV(\varepsilon)=\Inn\SV(\varepsilon)\oplus\Span{D_1}\oplus\Span{D_2}\oplus\Span{D_3}$,
where $\Inn\SV(\varepsilon)=\{\ad_x\mid x\in\SV(\varepsilon)\}$ and $D_1,D_2,D_3$ are the
outer derivations defined by
\[
\begin{aligned}
&D_1(L_m)=M_m, && D_1(Y_n)=0, && D_1(M_m)=0,\\
&D_2(L_m)=mM_m, && D_2(Y_n)=0, && D_2(M_m)=0,\\
&D_3(L_m)=0, && D_3(Y_n)=Y_n, && D_3(M_m)=2M_m.
\end{aligned}
\]
\end{theorem}

Theorem \ref{thm:der} will also be deduced from our results.

For every $t\in\Z$ and $s\in\varepsilon+\Z$ the inner derivations act as follows:
\begin{equation}\label{eq:ad}
\begin{aligned}
&\ad_{L_t}(L_m)=(t-m)L_{m+t}, && \ad_{L_t}(Y_n)=\Big(\frac t2-n\Big)Y_{n+t},
&& \ad_{L_t}(M_m)=-mM_{m+t},\\
&\ad_{M_t}(L_m)=tM_{m+t}, && \ad_{M_t}(Y_n)=0, && \ad_{M_t}(M_m)=0,\\
&\ad_{Y_s}(L_m)=\Big(s-\frac m2\Big)Y_{m+s}, && \ad_{Y_s}(Y_n)=(s-n)M_{n+s},
&& \ad_{Y_s}(M_m)=0.
\end{aligned}
\end{equation}

The descriptions of $\frac12$-derivations and of transposed Poisson structures of
$\SV(\varepsilon)$ were obtained by Yuan and Hua \cite{YH22}. Both of them will also be
deduced from our results (Corollaries \ref{cor:delta} and \ref{cor:tp} below).

\begin{theorem}[see \cite{YH22}]\label{thm:half}
Every $\frac12$-derivation of $\SV(\varepsilon)$ is trivial, i.e.,
$\Delta_{\frac{1}{2}}\big(\SV(\varepsilon)\big)=\Span{\Id}$. Moreover, all transposed Poisson structures on
$\SV(\varepsilon)$ are trivial.
\end{theorem}

\subsection{Three functional equations}\label{section3}

In this subsection we solve three functional equations arising in the study of homogeneous
quasi-derivations of $\SV(\varepsilon)$; they describe the action of the Witt part
$\Span{L_m\mid m\in\Z}$ on the parts $\Span{L_m}$, $\Span{Y_n}$ and $\Span{M_m}$. For a
sequence $\{u_m\}_{m\in\Z}$ of complex numbers, let
\[
w_m=u_{m+2}-2u_{m+1}+u_m
\]
be its second differences. We shall use two elementary facts: $w_m=0$ for all $m$ if and only
if $u_m=\alpha+\beta m$ for some $\alpha,\beta\in\C$; and the third differences
\begin{center}
$w_{m+1}-w_m=u_{m+3}-3u_{m+2}+3u_{m+1}-u_m$
\end{center} 
vanish for all $m$ if and only if
$u_m=\alpha+\beta m+\gamma m^2$ for some $\alpha,\beta,\gamma\in\C$.

If $u_{m+1}-u_m=c$ for all $m$, induction in both directions gives $u_m=u_0+cm$.
Applying this to $d_m=u_{m+1}-u_m$, which satisfies $d_{m+1}-d_m=w_m$, proves the first fact.
For the second, vanishing third differences give $w_m=w_0=:2\gamma$ for all $m$; since
$\gamma m^2$ has constant second differences $2\gamma$, the sequence $u_m-\gamma m^2$ has
vanishing second differences, and the first fact applies. The converses are direct computations.

\begin{lemma}\label{lem:witt}
Let $j\in\C$ and let $\{a_m\}_{m\in\Z}$, $\{a'_m\}_{m\in\Z}$ be complex numbers such that
\begin{equation}\label{eq:witt}
(m-n+j)a_m+(m-n-j)a_n=(m-n)a'_{m+n}\quad\text{for all } m,n\in\Z.
\end{equation}
Then there are $\alpha,\beta\in\C$ such that $a_m=\alpha+\beta m$ and
$a'_m=2\alpha+\beta(m+j)$ for all $m\in\Z$.
\end{lemma}

\begin{proof}
Putting $n=m+k$ in \eqref{eq:witt}, we obtain
\begin{equation}\label{eq:witt-k}
k\,a'_{2m+k}=(k-j)a_m+(k+j)a_{m+k}\quad\text{for all } m,k\in\Z.
\end{equation}
Taking $k=3$ in \eqref{eq:witt-k}, and then replacing $m$ by $m+1$ and taking $k=1$, we get
$3a'_{2m+3}=(3-j)a_m+(3+j)a_{m+3}$ and $a'_{2m+3}=(1-j)a_{m+1}+(1+j)a_{m+2}$. Eliminating
$a'_{2m+3}$, we obtain
\[
(3+j)a_{m+3}-3(1+j)a_{m+2}-3(1-j)a_{m+1}+(3-j)a_m=0.
\]
In terms of the second differences $w_m=a_{m+2}-2a_{m+1}+a_m$ this relation reads
\begin{equation}\label{eq:R3}
(3+j)w_{m+1}+(3-j)w_m=0 \quad\text{for all } m\in\Z.
\end{equation}
Similarly, taking $k=4$ in \eqref{eq:witt-k}, and then replacing $m$ by $m+1$ and taking
$k=2$, we get
\begin{center}
$4a'_{2m+4}=(4-j)a_m+(4+j)a_{m+4}$ and
$2a'_{2m+4}=(2-j)a_{m+1}+(2+j)a_{m+3}$,
\end{center} which gives
$(4+j)a_{m+4}-2(2+j)a_{m+3}-2(2-j)a_{m+1}+(4-j)a_m=0$, i.e.,
\begin{equation}\label{eq:R4}
(4+j)w_{m+2}+4w_{m+1}+(4-j)w_m=0 \quad\text{for all } m\in\Z.
\end{equation}
Now multiply \eqref{eq:R4} by $(3+j)^2$ and use \eqref{eq:R3} (for $m$ and for $m+1$) in
the form 
\begin{center}
$(3+j)w_{m+1}=(j-3)w_m$ and $(3+j)^2w_{m+2}=(j-3)^2w_m$.
\end{center} We obtain
\[
\big((4+j)(j-3)^2+4(j+3)(j-3)+(4-j)(j+3)^2\big)w_m=0.
\]
A direct computation shows that the expression in brackets is equal to $36$ (for every $j\in\C$, including $j=-3$). Hence
$w_m=0$ for all $m\in\Z$, and there are $\alpha,\beta\in\C$ with $a_m=\alpha+\beta m$.

Setting $n=0$ in \eqref{eq:witt}, we have
\begin{center}$ma'_m=(m+j)(\alpha+\beta m)+(m-j)\alpha=m\big(2\alpha+\beta(m+j)\big)$,\end{center} which gives
$a'_m=2\alpha+\beta(m+j)$ for $m\ne0$. Setting $m=1,\ n=-1$ in \eqref{eq:witt}, we have
\begin{center}$2a'_0=(2+j)(\alpha+\beta)+(2-j)(\alpha-\beta)=4\alpha+2j\beta$, i.e.,
$a'_0=2\alpha+\beta j$.\end{center}
\end{proof}

\begin{lemma}\label{lem:half}
Let $s\in\C$ and let $\{f_m\}_{m\in\Z}$, $\{f'_m\}_{m\in\Z}$ be complex numbers such that
\begin{equation}\label{eq:half}
\Big(m+s-\frac n2\Big)f_m+\Big(\frac m2-n-s\Big)f_n=(m-n)f'_{m+n}\quad
\text{for all } m,n\in\Z.
\end{equation}
Then there are $\phi,\psi\in\C$ such that $f_m=\phi+\psi m$ and
$f'_m=\frac32\phi+\psi(m+s)$ for all $m\in\Z$.
\end{lemma}

\begin{proof}
For $m\in\Z$ we set $\tau=\tau_m=\frac m2+s$. Putting $n=m+k$ in \eqref{eq:half}, we obtain
\begin{equation}\label{eq:half-k}
\Big(\tau_m-\frac k2\Big)f_m-(\tau_m+k)f_{m+k}=-k\,f'_{2m+k}
\quad\text{for all } m,k\in\Z.
\end{equation}
Taking $k=3$ in \eqref{eq:half-k}, and then replacing $m$ by $m+1$ (note that
$\tau_{m+1}=\tau+\frac12$) and taking $k=1$, we have
\[
\Big(\tau-\frac32\Big)f_m-(\tau+3)f_{m+3}=-3f'_{2m+3},\qquad
\tau f_{m+1}-\Big(\tau+\frac32\Big)f_{m+2}=-f'_{2m+3}.
\]
Eliminating $f'_{2m+3}$, we get
$(\tau+3)f_{m+3}-3(\tau+\frac32)f_{m+2}+3\tau f_{m+1}-(\tau-\frac32)f_m=0$. In terms of the
second differences $w_m=f_{m+2}-2f_{m+1}+f_m$ this relation reads
\begin{equation}\label{eq:hR3}
(\tau+3)w_{m+1}-\Big(\tau-\frac32\Big)w_m=0\quad\text{for all } m\in\Z.
\end{equation}
Similarly, taking $k=4$ in \eqref{eq:half-k}, and then replacing $m$ by $m+1$ and taking
$k=2$, we have
\[
(\tau-2)f_m-(\tau+4)f_{m+4}=-4f'_{2m+4},\qquad
\Big(\tau-\frac12\Big)f_{m+1}-\Big(\tau+\frac52\Big)f_{m+3}=-2f'_{2m+4},
\]
which gives
$(\tau-2)f_m-2(\tau-\frac12)f_{m+1}+2(\tau+\frac52)f_{m+3}-(\tau+4)f_{m+4}=0$, i.e.,
\begin{equation}\label{eq:hR4}
-(\tau+4)w_{m+2}-3w_{m+1}+(\tau-2)w_m=0\quad\text{for all } m\in\Z.
\end{equation}
Replacing $m$ by $m+1$ in \eqref{eq:hR3}, we have
$(\tau+\frac72)w_{m+2}=(\tau-1)w_{m+1}$. Now multiply \eqref{eq:hR4} by
$(\tau+\frac72)(\tau+3)$ and use $(\tau+\frac72)(\tau+3)w_{m+2}=(\tau-1)(\tau-\frac32)w_m$
and $(\tau+3)w_{m+1}=(\tau-\frac32)w_m$. We obtain
\[
0=\Big((\tau+4)(\tau-1)\Big(\tau-\frac32\Big)+3\Big(\tau+\frac72\Big)\Big(\tau-\frac32\Big)
-(\tau-2)\Big(\tau+\frac72\Big)(\tau+3)\Big)w_m=\frac{45}{4}w_m.
\]

 Hence $w_m=0$ for all
$m\in\Z$, and there are $\phi,\psi\in\C$ with $f_m=\phi+\psi m$.

Setting $n=0$ in \eqref{eq:half}, we get
$mf'_m=(m+s)(\phi+\psi m)+(\frac m2-s)\phi=m\big(\frac32\phi+\psi(m+s)\big)$, which gives
the required formula for $f'_m$ with $m\neq 0$. Setting $m=1,\ n=-1$ in \eqref{eq:half}, we
get $2f'_0=(\frac32+s)(\phi+\psi)+(\frac32-s)(\phi-\psi)=3\phi+2s\psi$, i.e.,
$f'_0=\frac32\phi+\psi s$.
\end{proof}

\begin{lemma}\label{lem:zero}
Let $j\in\Z$ and let $\{b_m\}_{m\in\Z}$, $\{b'_m\}_{m\in\Z}$ be complex numbers such that
\begin{equation}\label{eq:zero}
(m+j)b_m-(n+j)b_n=(m-n)b'_{m+n}\quad\text{for all } m,n\in\Z.
\end{equation}
Then there are $\beta_1,\beta_2\in\C$ such that $b_m=\beta_1+\beta_2m$ for
$m\neq -j$ and $b'_m=\beta_1+\beta_2(m+j)$ for all $m\in\Z$.
\end{lemma}

\begin{proof}
Set $g_m=(m+j)b_m$, so that $g_m-g_n=(m-n)b'_{m+n}$. Taking $(m,n)=(m,m+3)$ and
$(m,n)=(m+1,m+2)$, we get $g_m-g_{m+3}=-3b'_{2m+3}$ and $g_{m+1}-g_{m+2}=-b'_{2m+3}$.
Hence \begin{center}$g_{m+3}-3g_{m+2}+3g_{m+1}-g_m=0$ for all $m\in\Z$,\end{center} and there are $P,Q,R\in\C$ such
that $g_m=P+Qm+Rm^2$ for all $m\in\Z$. Then\begin{center}
$(m-n)b'_{m+n}=g_m-g_n=(m-n)\big(Q+R(m+n)\big)$.\end{center} For every $k\in\Z$ we can choose $m\ne n$
with $m+n=k$, so $b'_k=Q+Rk$. Finally, $0=g_{-j}=P-Qj+Rj^2$, hence $g_m=(m+j)(Q-Rj+Rm)$.
Setting $\beta_1=Q-Rj$ and $\beta_2=R$, we get
\begin{center}$(m+j)b_m=(m+j)(\beta_1+\beta_2m)$ for all
$m\in\Z$ and $b'_m=\beta_1+\beta_2(m+j)$.\end{center}
\end{proof}

\subsection{Quasi-derivations of the Schr\"odinger--Witt algebras}\label{section4}

\begin{definition}\label{def:new}
For $t\in\Z$ and $s\in\varepsilon+\Z$ we define the linear maps
$\Omega_t,\Psi_t,\Theta_s$ and $\Omega'_t,\Psi'_t,\Theta'_s$ of $\SV(\varepsilon)$ by
\[
\begin{aligned}
&\Omega_t(L_m)=L_{m+t}, && \Omega_t(Y_n)=\tfrac12Y_{n+t}, && \Omega_t(M_m)=0,\\
&\Omega'_t(L_m)=2L_{m+t}, && \Omega'_t(Y_n)=\tfrac32Y_{n+t}, && \Omega'_t(M_m)=M_{m+t},\\
&\Psi_t(L_m)=mM_{m+t}, && \Psi_t(Y_n)=0, && \Psi_t(M_m)=0,\\
&\Psi'_t(L_m)=(m+t)M_{m+t}, && \Psi'_t(Y_n)=0, && \Psi'_t(M_m)=0,\\
&\Theta_s(L_m)=Y_{m+s}, && \Theta_s(Y_n)=M_{n+s}, && \Theta_s(M_m)=0,\\
&\Theta'_s(L_m)=\tfrac32Y_{m+s}, && \Theta'_s(Y_n)=2M_{n+s}, && \Theta'_s(M_m)=0.
\end{aligned}
\]
\end{definition}

Note that $\Omega_0=\Id-\frac12D_3$ and $\Psi_0=D_2$. The maps $\Omega_t$ and $\Psi_t$ are
even of degree $t$, and the map $\Theta_s$ is odd of degree $s$.

\begin{lemma}\label{lem:new}
For all $t\in\Z$ and $s\in\varepsilon+\Z$, the maps $\Omega_t$, $\Psi_t$ and $\Theta_s$ are
quasi-derivations of $\SV(\varepsilon)$ with associated maps $\Omega'_t$, $\Psi'_t$ and
$\Theta'_s$, respectively. The maps $\Omega_t$, $\Theta_s$ and, for $t\ne0$, $\Psi_t$ are
not derivations.
\end{lemma}

\begin{proof}
It is enough to verify \eqref{eq:qd} on pairs of basis elements. For $\Omega_t$ we have
\begin{align*}
[\Omega_t(L_m),L_n]+[L_m,\Omega_t(L_n)]&=(m+t-n)L_{m+n+t}+(m-n-t)L_{m+n+t}\\
&=2(m-n)L_{m+n+t}=\Omega'_t\big((m-n)L_{m+n}\big),\\
[\Omega_t(L_m),Y_n]+[L_m,\Omega_t(Y_n)]&=\Big(\frac{m+t}2-n\Big)Y_{m+n+t}
+\frac12\Big(\frac m2-n-t\Big)Y_{m+n+t}\\
&=\frac32\Big(\frac m2-n\Big)Y_{m+n+t}=\Omega'_t\Big(\Big(\frac m2-n\Big)Y_{m+n}\Big),\\
[\Omega_t(L_m),M_n]+[L_m,\Omega_t(M_n)]&=-nM_{m+n+t}=\Omega'_t(-nM_{m+n}),\\
[\Omega_t(Y_m),Y_n]+[Y_m,\Omega_t(Y_n)]&=\frac12(m+t-n)M_{m+n+t}+\frac12(m-n-t)M_{m+n+t}\\
&=(m-n)M_{m+n+t}=\Omega'_t\big((m-n)M_{m+n}\big),
\end{align*}
and for the pairs $(Y_m,M_n)$ and $(M_m,M_n)$ both sides of \eqref{eq:qd} are zero. For
$\Psi_t$ the only nontrivial case is
\begin{align*}
[\Psi_t(L_m),L_n]+[L_m,\Psi_t(L_n)]&=m(m+t)M_{m+n+t}-n(n+t)M_{m+n+t}\\
&=(m-n)(m+n+t)M_{m+n+t}=\Psi'_t\big((m-n)L_{m+n}\big);
\end{align*}
for the pairs $(L_m,Y_n)$ and $(L_m,M_n)$ both sides are zero because
$\Psi'_t(Y_n)=\Psi'_t(M_n)=0$ and $[M_k,Y_n]=[M_k,M_n]=0$, and for all other pairs both
sides are obviously zero. For $\Theta_s$ we have
\begin{align*}
[\Theta_s(L_m),L_n]+[L_m,\Theta_s(L_n)]&=\Big(m+s-\frac n2\Big)Y_{m+n+s}
+\Big(\frac m2-n-s\Big)Y_{m+n+s}\\
&=\frac32(m-n)Y_{m+n+s}=\Theta'_s\big((m-n)L_{m+n}\big),\\
[\Theta_s(L_m),Y_n]+[L_m,\Theta_s(Y_n)]&=(m+s-n)M_{m+n+s}-(n+s)M_{m+n+s}\\
&=2\Big(\frac m2-n\Big)M_{m+n+s}=\Theta'_s\Big(\Big(\frac m2-n\Big)Y_{m+n}\Big),
\end{align*}
and for all other pairs of basis elements both sides of \eqref{eq:qd} are zero (we use
$\Theta'_s(M_k)=0$ and $[M_k,Y_n]=[M_k,M_n]=0$). By Lemmas \ref{lem:unique} and
\ref{lem:perfect}, a quasi-derivation $D$ of $\SV(\varepsilon)$ is a derivation if and only
if $D'=D$. Since \begin{center}$\Omega'_t(M_m)\neq\Omega_t(M_m)$, $\Theta'_s(L_m)\ne\Theta_s(L_m)$ and
$\Psi'_t(L_0)=tM_t\ne0=\Psi_t(L_0)$ for $t\neq0$,
\end{center} the last statement follows.
\end{proof}

\begin{proposition}\label{prop:even}
Let $D$ be an even quasi-derivation of degree $j\in\Z$ of $\SV(\varepsilon)$.
\begin{enumerate}[{\rm (1)}]
\item If $j\ne0$, then there are $u,v,\mu,\nu\in\C$ and a central map $\zeta$ with
$\zeta(x)=0$ for all basis elements $x\notin\{L_{-j},M_{-j}\}$ such that
\[
D=u\ad_{L_j}+v\ad_{M_j}+\mu\,\Omega_j+\nu\,\Psi_j+\zeta.
\]
\item If $j=0$, then there are $u,\lambda_1,\lambda_2,\lambda_3,\lambda\in\C$ and a central
map $\zeta$ with $\zeta(x)=0$ for all basis elements $x\notin\{L_0,M_0\}$ such that
\[
D=u\ad_{L_0}+\lambda_1D_1+\lambda_2D_2+\lambda_3D_3+\lambda\Id+\zeta.
\]
\end{enumerate}
\end{proposition}

\begin{proof}
Since $D$ and $D'$ are even of degree $j$, we can write
\[
\begin{aligned}
&D(L_m)=a_mL_{m+j}+b_mM_{m+j}, && D(Y_n)=c_nY_{n+j}, && D(M_m)=d_mL_{m+j}+e_mM_{m+j},\\
&D'(L_m)=a'_mL_{m+j}+b'_mM_{m+j}, && D'(Y_n)=c'_nY_{n+j}, &&
D'(M_m)=d'_mL_{m+j}+e'_mM_{m+j},
\end{aligned}
\]
where all coefficients are complex numbers. Let us apply \eqref{eq:qd} to the pairs of
basis elements. For example, for the pair $(L_m,Y_n)$ we have
\[
[D(L_m),Y_n]+[L_m,D(Y_n)]=\Big(\Big(\frac{m+j}2-n\Big)a_m
+\Big(\frac m2-n-j\Big)c_n\Big)Y_{m+n+j}
\]
and $D'([L_m,Y_n])=(\frac m2-n)c'_{m+n}Y_{m+n+j}$. Considering in the same way the pairs
$(L_m,L_n)$, $(L_m,Y_n)$, $(L_m,M_n)$, $(Y_m,Y_n)$, $(Y_m,M_n)$ and $(M_m,M_n)$ and
comparing the coefficients at the basis elements, we obtain the following relations:
\begin{align}
&(m-n+j)a_m+(m-n-j)a_n=(m-n)a'_{m+n}, && m,n\in\mathbb{Z}, \label{e1}\\
&(m+j)b_m-(n+j)b_n=(m-n)b'_{m+n}, && m,n\in\mathbb{Z}, \label{e2}\\
&\Big(\frac{m+j}2-n\Big)a_m+\Big(\frac m2-n-j\Big)c_n=\Big(\frac m2-n\Big)c'_{m+n}, && m\in\mathbb{Z},\, n\in\varepsilon+\mathbb{Z}, \label{e3}\\
&(m-n-j)d_n+nd'_{m+n}=0, && m,n\in\mathbb{Z}, \label{e4}\\
&(n+j)e_n+na_m=ne'_{m+n}, && m,n\in\mathbb{Z}, \label{e5}\\
&(m-n)d'_{m+n}=0, \quad (m-n+j)c_m+(m-n-j)c_n=(m-n)e'_{m+n}, && m,n\in\varepsilon+\mathbb{Z}, \label{e6}\\
&\Big(m-\frac{n+j}2\Big)d_n=0, && m\in\varepsilon+\mathbb{Z},\, n\in\mathbb{Z}, \label{e7_1}\\
&md_n-nd_m=0, && m,n\in\mathbb{Z}. \label{e7_2}
\end{align}

\smallskip
\begin{enumerate}
    \item[\bf{Step 1.}] Taking two different values of $m\in\varepsilon+\Z$ in
\eqref{e7_1}, we get $d_n=0$ for all $n\in\Z$. Then \eqref{e4} with $n=1$ gives $d'_k=0$ for
all $k\in\Z$.

\smallskip
    \item[\bf{Step 2.}]  By Lemma \ref{lem:witt} and \eqref{e1}, there are $\alpha,\beta\in\C$ such
that
\[
a_m=\alpha+\beta m,\qquad a'_m=2\alpha+\beta(m+j)\quad\text{for all } m\in\Z.
\]
By Lemma \ref{lem:zero} and \eqref{e2}, there are $\beta_1,\beta_2\in\C$ such that
$b_m=\beta_1+\beta_2m$ for $m\ne-j$ and $b'_m=\beta_1+\beta_2(m+j)$ for all $m\in\Z$.

\smallskip
    \item[\bf{Step 3.}]  Setting $n=1$ in \eqref{e5}, we have $e'_{m+1}=\alpha+\beta m+(1+j)e_1$.
Hence, if we define $\rho\in\C$ by $2\rho=(1+j)(e_1-\beta)$, then
\begin{equation}\label{e8}
e'_k=\alpha+\beta(k+j)+2\rho\quad\text{for all } k\in\Z.
\end{equation}
Setting $n=-j$ in \eqref{e5} and using \eqref{e8}, we get $0=j(e'_{m-j}-a_m)=2j\rho$,
i.e.,
\begin{equation}\label{e9}
j\rho=0.
\end{equation}
Setting $m=0$ in \eqref{e5} and using \eqref{e8}, \eqref{e9}, we get\begin{center}
$(n+j)e_n=n\big(\beta(n+j)+2\rho\big)=(n+j)(\beta n+2\rho)-2j\rho
=(n+j)(\beta n+2\rho)$.
\end{center} Hence $e_n=\beta n+2\rho$ for all $n\ne-j$.

\smallskip
    \item[\bf{Step 4.}]  Let us set $\tilde c_n=c_n-\frac\alpha2-\beta n$ for $n\in\varepsilon+\Z$. A
direct computation shows that\begin{center}
$(m-n+j)(\frac\alpha2+\beta m)+(m-n-j)(\frac\alpha2+\beta n)=(m-n)(\alpha+\beta(m+n+j))$.
\end{center}
Hence, by the second relation of \eqref{e6} and \eqref{e8},
\[
(m-n+j)\tilde c_m+(m-n-j)\tilde c_n=2\rho(m-n)\quad\text{for all } m,n\in\varepsilon+\Z.
\]
Writing $m=p+\varepsilon$, $n=q+\varepsilon$ with $p,q\in\Z$ and applying Lemma
\ref{lem:witt} to the sequences $\{\tilde c_{p+\varepsilon}\}_{p\in\Z}$ and
$\{2\rho\}_{k\in\Z}$, we obtain $\alpha_0,\beta_0\in\C$ such that
$\tilde c_{p+\varepsilon}=\alpha_0+\beta_0p$ and $2\rho=2\alpha_0+\beta_0(k+j)$ for all
$p,k\in\Z$. Therefore $\beta_0=0$ and $\alpha_0=\rho$, that is,
\[
c_n=\frac\alpha2+\beta n+\rho\quad\text{for all } n\in\varepsilon+\Z.
\]
Now substituting $a_m$ and $c_n$ in \eqref{e3}, a direct computation gives
\[
\Big(n-\frac m2\Big)\Big(c'_{m+n}-\frac32\alpha-\beta(m+n+j)-\rho\Big)=j\rho=0.
\]
For every $r\in\varepsilon+\Z$ we can choose $n\in\varepsilon+\Z$ with $r\ne3n$ and put
$m=r-n$, so that $\frac m2-n=\frac{r-3n}2\neq0$. Hence
$c'_r=\frac32\alpha+\beta(r+j)+\rho$ for all $r\in\varepsilon+\Z$.

\smallskip
    \item[\bf{Step 5.}]  Now we compare $D$ with known quasi-derivations, using \eqref{eq:ad}.

Let $j\ne0$. Then $\rho=0$ by \eqref{e9}. Consider the quasi-derivation
\[
\widehat D=(\alpha+\beta j)\Omega_j-\beta\ad_{L_j}+\frac{\beta_1}{j}\ad_{M_j}+\beta_2\Psi_j.
\]
Then
\begin{center}
$\widehat D(L_m)=(\alpha+\beta m)L_{m+j}+(\beta_1+\beta_2m)M_{m+j},$\\
$\widehat D(Y_n)=\Big(\frac\alpha2+\beta n\Big)Y_{n+j},$\quad
$\widehat D(M_m)=\beta mM_{m+j}.$
\end{center}
By Steps 1--4, the map $\zeta=D-\widehat D$ vanishes on all basis elements except $L_{-j}$
and $M_{-j}$, and $\zeta(L_{-j}),\zeta(M_{-j})\in\Span{M_0}$. Hence $\zeta$ is a central map
and the first statement holds with $u=-\beta$, $v=\frac{\beta_1}{j}$, $\mu=\alpha+\beta j$,
$\nu=\beta_2$.

Let $j=0$. Consider the quasi-derivation
$\widehat D=\alpha\,\Omega_0-\beta\ad_{L_0}+\beta_1D_1+\beta_2D_2+\rho D_3$. Then
\begin{center}$\widehat D(L_m)=(\alpha+\beta m)L_m+(\beta_1+\beta_2m)M_m$,\\ 
$\widehat D(Y_n)=(\frac\alpha2+\beta n+\rho)Y_n$ and
$\widehat D(M_m)=(\beta m+2\rho)M_m$.
\end{center} As above, $\zeta=D-\widehat D$ is a central map
vanishing on all basis elements except $L_0$ and $M_0$. Since $\Omega_0=\Id-\frac12D_3$, we
get $D=-\beta\ad_{L_0}+\beta_1D_1+\beta_2D_2+(\rho-\frac\alpha2)D_3+\alpha\Id+\zeta$.
\end{enumerate}
\end{proof}

\begin{proposition}\label{prop:odd}
Let $D$ be an odd quasi-derivation of degree $s\in\varepsilon+\Z$ of $\SV(\varepsilon)$.
Then there are $w,\gamma\in\C$ and a central map $\zeta$ with $\zeta(x)=0$ for all basis
elements $x\ne Y_{-s}$ such that
\[
D=w\ad_{Y_s}+\gamma\,\Theta_s+\zeta.
\]
\end{proposition}

\begin{proof}
Since $D$ and $D'$ are odd of degree $s$, we can write
\[
\begin{aligned}
&D(L_m)=f_mY_{m+s}, && D(M_m)=g_mY_{m+s}, && D(Y_n)=h_nL_{n+s}+k_nM_{n+s},\\
&D'(L_m)=f'_mY_{m+s}, && D'(M_m)=g'_mY_{m+s}, && D'(Y_n)=h'_nL_{n+s}+k'_nM_{n+s}.
\end{aligned}
\]
Applying \eqref{eq:qd} to the pairs $(L_m,L_n)$, $(L_m,Y_n)$, $(L_m,M_n)$, $(Y_m,Y_n)$ and
$(Y_m,M_n)$ and comparing the coefficients at the basis elements, we obtain the following
relations (the pair $(M_m,M_n)$ gives no relations):
\begin{align}
&\Big(m+s-\frac n2\Big)f_m+\Big(\frac m2-n-s\Big)f_n=(m-n)f'_{m+n}, && m,n\in\mathbb{Z}, \label{o1}\\
&(m-n-s)h_n=\Big(\frac m2-n\Big)h'_{m+n}, && m\in\mathbb{Z},\ n\in\varepsilon+\mathbb{Z}, \label{o2}\\
&(m+s-n)f_m-(n+s)k_n=\Big(\frac m2-n\Big)k'_{m+n}, && m\in\mathbb{Z},\ n\in\varepsilon+\mathbb{Z}, \label{o3}\\
&\Big(\frac m2-n-s\Big)g_n=-ng'_{m+n}, && m,n\in\mathbb{Z}, \label{o4}\\
&\Big(\frac{m+s}2-n\Big)h_m+\Big(m-\frac{n+s}2\Big)h_n=(m-n)g'_{m+n}, && m,n\in\varepsilon+\mathbb{Z}, \label{o5}\\
&(m-n-s)g_n-nh_m=0, && m\in\varepsilon+\mathbb{Z},\ n\in\mathbb{Z}. \label{o6}
\end{align}

\smallskip
\begin{enumerate}
    \item[\bf{Step 1.}] Setting $n=0$ in \eqref{o6}, we have $(m-s)g_0=0$ for all
$m\in\varepsilon+\Z$, so $g_0=0$. Setting $n=1$ and $n=2$ in \eqref{o6}, we have
$h_m=(m-1-s)g_1$ and $2h_m=(m-2-s)g_2$. Hence 
\begin{center}$2(m-1-s)g_1=(m-2-s)g_2$ for all
$m\in\varepsilon+\Z$.\end{center} Comparing the coefficients at $m$ and the constant terms of these
polynomials of degree at most one in $m$ (which coincide at infinitely many points), we get
$2g_1=g_2$ and $2(1+s)g_1=(2+s)g_2=2(2+s)g_1$, so $g_1=0$. Therefore $h_m=0$ for all
$m\in\varepsilon+\Z$. Now \eqref{o6} gives $(m-n-s)g_n=0$ for all $m\in\varepsilon+\Z$,
whence $g_n=0$ for all $n\in\Z$.

\smallskip
\item[\bf{Step 2.}] Since $h_n=0$, the relation \eqref{o2} gives $(\frac m2-n)h'_{m+n}=0$. For
every $r\in\varepsilon+\Z$ we choose $n\in\varepsilon+\Z$ with $r\ne3n$ and put $m=r-n$; then
$\frac m2-n\ne0$ and $h'_r=0$. Similarly, \eqref{o5} gives $(m-n)g'_{m+n}=0$ for all
$m,n\in\varepsilon+\Z$, and since every $k\in\Z$ can be written as $k=m+n$ with $m\ne n$,
$m,n\in\varepsilon+\Z$, we get $g'_k=0$ for all $k\in\Z$. Now \eqref{o4} holds trivially.

\smallskip
\item[\bf{Step 3.}] By Lemma \ref{lem:half} and \eqref{o1}, there are $\phi,\psi\in\C$ such that
\[
f_m=\phi+\psi m,\qquad f'_m=\frac32\phi+\psi(m+s)\quad\text{for all } m\in\Z.
\]

\smallskip
\item[\bf{Step 4.}] A direct computation shows that\begin{center}
$(m+s-n)(\phi+\psi m)-(n+s)(\phi+2\psi n)=(\frac m2-n)\big(2\phi+2\psi(m+n+s)\big)$.\end{center}
Hence \eqref{o3} can be written as
\[
(n+s)\big(k_n-\phi-2\psi n\big)=\Big(n-\frac m2\Big)\big(k'_{m+n}-2\phi-2\psi(m+n+s)\big).
\]
Taking $m=2n\in\Z$, we get $(n+s)(k_n-\phi-2\psi n)=0$, i.e., $k_n=\phi+2\psi n$ for all
$n\ne-s$. Then the right hand side vanishes for all $m\in\Z$, $n\in\varepsilon+\Z$, and, as
in Step 2, we obtain $k'_r=2\phi+2\psi(r+s)$ for all $r\in\varepsilon+\Z$.

\smallskip
\item[\bf{Step 5.}] Let $\gamma=\phi+2\psi s$ and $w=-2\psi$. By \eqref{eq:ad},
\[
(w\ad_{Y_s}+\gamma\Theta_s)(L_m)=\Big(w\Big(s-\frac m2\Big)+\gamma\Big)Y_{m+s}
=(\phi+\psi m)Y_{m+s},
\]
\[
(w\ad_{Y_s}+\gamma\Theta_s)(Y_n)=\big(w(s-n)+\gamma\big)M_{n+s}=(\phi+2\psi n)M_{n+s},
\]
and $(w\ad_{Y_s}+\gamma\Theta_s)(M_m)=0$. By Steps 1--4, the map
$\zeta=D-w\ad_{Y_s}-\gamma\Theta_s$ vanishes on all basis elements except $Y_{-s}$, and
$\zeta(Y_{-s})=(k_{-s}-\phi+2\psi s)M_0\in\Span{M_0}$. Hence $\zeta$ is a central map.
\end{enumerate}
\end{proof}

Now we are ready to formulate and prove the main result of the paper.

\begin{theorem}\label{thm:main}
A linear map $D$ of $\SV(\varepsilon)$ is a quasi-derivation if and only if there are a
derivation $d\in\Der\big(\SV(\varepsilon)\big)$, a complex number $\lambda$, three finitely supported families of
complex numbers $\{\alpha_t\}_{t\in\Z\setminus\{0\}}$, $\{\beta_t\}_{t\in\Z\setminus\{0\}}$,
$\{\gamma_s\}_{s\in\varepsilon+\Z}$ and a central map
$\zeta:\SV(\varepsilon)\to\Span{M_0}$ such that
\begin{equation}\label{eq:main}
D=d+\lambda\Id+\sum\limits_{t\ne0}\alpha_t\Omega_t+\sum\limits_{t\ne0}\beta_t\Psi_t
+\sum\limits_{s\in\varepsilon+\Z}\gamma_s\Theta_s+\zeta.
\end{equation}
Explicitly, this means that
\[
\begin{aligned}
D(L_m)&=d(L_m)+\lambda L_m+\sum\limits_{t\ne0}\alpha_tL_{m+t}+\sum\limits_{t\ne0}\beta_t\,m\,M_{m+t}
+\sum\limits_{s\in\varepsilon+\Z}\gamma_sY_{m+s}+\zeta(L_m),\\
D(Y_n)&=d(Y_n)+\lambda Y_n+\frac12\sum\limits_{t\ne0}\alpha_tY_{n+t}
+\sum\limits_{s\in\varepsilon+\Z}\gamma_sM_{n+s}+\zeta(Y_n),\\
D(M_m)&=d(M_m)+\lambda M_m+\zeta(M_m).
\end{aligned}
\]
The decomposition \eqref{eq:main} is unique, and the map associated with $D$ is
\[
D'=d+2\lambda\Id+\sum\limits_{t\ne0}\alpha_t\Omega'_t+\sum\limits_{t\ne0}\beta_t\Psi'_t
+\sum\limits_{s\in\varepsilon+\Z}\gamma_s\Theta'_s.
\]
\end{theorem}

\begin{proof}
If $D$ is of the form \eqref{eq:main}, then $D$ is a quasi-derivation with the indicated
associated map by Lemmas \ref{lem:unique} and \ref{lem:new} (all sums are finite).

Conversely, let $D$ be a quasi-derivation. By Proposition \ref{prop:even}, for every
$j\in\Z\setminus\{0\}$ there are $u_j,v_j,\mu_j,\nu_j\in\C$ and a central map $\zeta_j$ such
that the even component of degree $j$ of $D$ is equal to
$u_j\ad_{L_j}+v_j\ad_{M_j}+\mu_j\Omega_j+\nu_j\Psi_j+\zeta_j$, and there are
$u_0,\lambda_1,\lambda_2,\lambda_3,\lambda\in\C$ and a central map $\zeta_0$ such that the
even component of degree $0$ of $D$ is equal to
\begin{center}$u_0\ad_{L_0}+\lambda_1D_1+\lambda_2D_2+\lambda_3D_3+\lambda\Id+\zeta_0$.
\end{center} By Proposition
\ref{prop:odd}, for every $s\in\varepsilon+\Z$ there are $w_s,\gamma_s\in\C$ and a central
map $\zeta'_s$ such that the odd component of degree $s$ of $D$ is equal to
$w_s\ad_{Y_s}+\gamma_s\Theta_s+\zeta'_s$.

Let us show that only finitely many of the numbers $u_j,v_j,\mu_j,\nu_j,w_s,\gamma_s$ are
nonzero. By \eqref{eq:ad}, for $j\neq0$ and $m+j\ne0$ the coefficients of $D(L_m)$ at
$L_{m+j}$ and at $M_{m+j}$ are equal to
\[
u_j(j-m)+\mu_j\qquad\text{and}\qquad v_jj+\nu_jm,
\]
respectively, and the coefficient of $D(L_m)$ at $Y_{m+s}$ is equal to
$w_s(s-\frac m2)+\gamma_s$. Since $D(L_0)$, $D(L_1)$ and $D(L_2)$ are finite linear
combinations of basis elements, for all but finitely many $j$ and $s$ we have\begin{center}
$u_jj+\mu_j=u_j(j-1)+\mu_j=0$, $v_jj+\nu_j=v_jj+2\nu_j=0$ and
$w_ss+\gamma_s=w_s(s-1)+\gamma_s=0$\end{center}  (we use $m=0,1$ for the first and the last pair of equalities and $m=1,2$ for the second one), which gives $u_j=\mu_j=v_j=\nu_j=0$ and
$w_s=\gamma_s=0$ for all but finitely many $j$ and $s$.

Consequently, $x=\sum\limits_{j}(u_jL_j+v_jM_j)+\sum\limits_sw_sY_s$ (with $v_0=0$) is an element of
$\SV(\varepsilon)$, and $d=\ad_x+\lambda_1D_1+\lambda_2D_2+\lambda_3D_3$ is a derivation. The
map
\[
\zeta=D-d-\lambda\Id-\sum\limits_{t\ne0}\mu_t\Omega_t-\sum\limits_{t\ne0}\nu_t\Psi_t-\sum\limits_s\gamma_s\Theta_s
\]
is linear, and for every basis element $y$ we have
$\zeta(y)=\sum\limits_{j}\zeta_j(y)+\sum\limits_s\zeta'_s(y)\in\Span{M_0}$. Hence $\zeta$ is a central map
and \eqref{eq:main} holds with $\alpha_t=\mu_t$ and $\beta_t=\nu_t$.

Let us prove the uniqueness. Assume that
$$d+\lambda\Id+\sum\limits_{t\ne0}\alpha_t\Omega_t+\sum\limits_{t\ne0}\beta_t\Psi_t
+\sum\limits_s\gamma_s\Theta_s+\zeta=0.$$

By Lemmas \ref{lem:unique} and \ref{lem:perfect}, the map
associated with the left hand side is zero, i.e.,
\begin{center}$d+2\lambda\Id+\sum\limits_{t\ne0}\alpha_t\Omega'_t+\sum\limits_{t\ne0}\beta_t\Psi'_t
+\sum\limits_s\gamma_s\Theta'_s=0$.
\end{center}Subtracting the first equality from the second one, we get
\[
\lambda\Id+\sum\limits_{t\ne0}\alpha_t(\Omega'_t-\Omega_t)+\sum\limits_{t\ne0}\beta_t(\Psi'_t-\Psi_t)
+\sum\limits_s\gamma_s(\Theta'_s-\Theta_s)-\zeta=0.
\]
Applying this map to $M_n$ with $n\ne0$ and $n\ne -t$ for all $t$ with $\alpha_t\neq0$, we
obtain $\lambda M_n+\sum\limits_{t\ne0}\alpha_tM_{n+t}\in\Span{M_0}$, whence $\lambda=0$ and
$\alpha_t=0$ for all $t$. Applying it to $L_m$, we obtain
$\sum\limits_{t\ne0}\beta_t\,t\,M_{m+t}+\frac12\sum\limits_s\gamma_sY_{m+s}\in\Span{M_0}$ for all $m\in\Z$,
whence $\beta_t=0$ and $\gamma_s=0$ for all $t$ and $s$. Then $d+\zeta=0$ and, by the second
equality, $d=0$; hence $\zeta=0$.
\end{proof}

\begin{corollary}\label{cor:decomp}
Let $\mathcal Q\big(\SV(\varepsilon)\big)$ be the space of all finite linear combinations of the
maps $\Omega_t,\Psi_t$ $(t\in\Z\setminus\{0\})$ and $\Theta_s$ $(s\in\varepsilon+\Z)$. Then
\[
\QDer\big(\SV(\varepsilon)\big)=\Der\big(\SV(\varepsilon)\big)\oplus\Delta_{\frac{1}{2}}\big(\SV(\varepsilon)\big)\oplus
\mathcal Q\big(\SV(\varepsilon)\big)\oplus\Hom\big(\SV(\varepsilon),Z(\SV(\varepsilon))\big).
\]
In particular,
\begin{center}$\QDer\big(\SV(\varepsilon)\big)$ is strictly larger than
$\Der\big(\SV(\varepsilon)\big)+
\Delta_{\frac{1}{2}}\big(\SV(\varepsilon)\big)+\Hom\big(\SV(\varepsilon),Z(\SV(\varepsilon))\big)$.\end{center}
\end{corollary}

\begin{proof}
This follows from Theorems \ref{thm:main} and \ref{thm:half}, Lemma \ref{lem:perfect} and
the uniqueness of the decomposition \eqref{eq:main}.
\end{proof}

\subsection{Corollaries}\label{section5}

\begin{corollary}\label{cor:delta}
Let $\varepsilon\in\{0,\frac12\}$ and $\delta\in\C$. Then
\[
\Delta_\delta\big(\SV(\varepsilon)\big)=
\begin{cases}
\Der\ \SV(\varepsilon), & \delta=1,\\
\Span{\Id}, & \delta=\frac12,\\
0, & \delta\ne\frac12,1.
\end{cases}
\]
In other words, $\SV(\varepsilon)$ has no nonzero $\delta$-derivations for
$\delta\notin\{\frac12,1\}$, every $\frac12$-derivation of $\SV(\varepsilon)$ is trivial (cf. Theorem
\ref{thm:half} and \cite{YH22}), and the
$1$-derivations are the derivations described in Theorem \ref{thm:der}.
\end{corollary}

\begin{proof}
The case $\delta=1$ is the definition. If $\delta=0$, then a $0$-derivation vanishes on
$[\SV(\varepsilon),\SV(\varepsilon)]=\SV(\varepsilon)$ (Lemma \ref{lem:perfect}). Let
$\delta\notin\{0,1\}$ and put $c=\delta^{-1}\neq1$. A linear map $\varphi$ is a
$\delta$-derivation if and only if it is a quasi-derivation with $\varphi'=c\varphi$, and by
Lemma \ref{lem:grading} every homogeneous component $D$ of $\varphi$ is a quasi-derivation
with $D'=cD$. We use Propositions \ref{prop:even} and \ref{prop:odd} together with Lemmas
\ref{lem:unique} and \ref{lem:new}.

Let $D$ be even of degree $j\ne0$, so that
\begin{center}$D=u\ad_{L_j}+v\ad_{M_j}+\mu\Omega_j+\nu\Psi_j+\zeta$ and
$D'=u\ad_{L_j}+v\ad_{M_j}+\mu\Omega'_j+\nu\Psi'_j$.\end{center} Comparing $D'(M_n)$ and $cD(M_n)$ for
$n\ne-j$, we get $-un+\mu=-cun$ for all $n\neq-j$, so $u=\mu=0$. Comparing the coefficients at
$M_{m+j}$ in $D'(L_m)$ and $cD(L_m)$ for $m\ne-j$, we get $vj+\nu(m+j)=c(vj+\nu m)$ for all
$m\ne-j$, so $\nu=0$ and $v=0$. Then $D=\zeta$ and $0=D'=c\zeta$, so $D=0$.

Let $D$ be even of degree $0$, so that
\begin{center}$D=u\ad_{L_0}+\lambda_1D_1+\lambda_2D_2+\lambda_3D_3+\lambda\Id+\zeta$ and
$D'=u\ad_{L_0}+\lambda_1D_1+\lambda_2D_2+\lambda_3D_3+2\lambda\Id$.\end{center} Comparing $D'(L_m)$ and
$cD(L_m)$ for $m\ne0$, we get
\begin{center}$-um+2\lambda=c(-um+\lambda)$ and
$\lambda_1+\lambda_2m=c(\lambda_1+\lambda_2m)$,
\end{center}so $u=\lambda_1=\lambda_2=0$ and
$(2-c)\lambda=0$. Comparing $D'(Y_n)$ and $cD(Y_n)$, we get
$\lambda_3+2\lambda=c(\lambda_3+\lambda)$, whence $(1-c)\lambda_3=0$ and $\lambda_3=0$. Then
$D=\lambda\Id+\zeta$, and $D'=2\lambda\Id=cD$ gives $c\zeta=(2-c)\lambda\Id=0$, so $\zeta=0$ and
$D=\lambda\Id$ with $\lambda=0$ unless $c=2$, i.e., $\delta=\frac12$.

Let $D$ be odd of degree $s$, so that $D=w\ad_{Y_s}+\gamma\Theta_s+\zeta$ and
$D'=w\ad_{Y_s}+\gamma\Theta'_s$. Comparing $D'(L_m)$ and $cD(L_m)$, we get
$w(s-\frac m2)+\frac32\gamma=c\big(w(s-\frac m2)+\gamma\big)$ for all $m\in\Z$, and comparing
$D'(Y_n)$ and $cD(Y_n)$ for $n\ne-s$, we get $w(s-n)+2\gamma=c\big(w(s-n)+\gamma\big)$. Hence
$w=0$, $(\frac32-c)\gamma=(2-c)\gamma=0$, so $\gamma=0$. Then $D=\zeta$ and $0=D'=c\zeta$, so $D=0$.

Since $\varphi$ is the sum of its homogeneous components, $\varphi=0$ if
$\delta\neq\frac12$ and $\varphi\in\Span{\Id}$ if $\delta=\frac12$. Conversely, $\Id$ is a
$\frac12$-derivation.
\end{proof}

The description of biderivations of $\SV(\varepsilon)$ is given in the following theorem.

\begin{theorem}[see \cite{LT19}]\label{thm:LTbider}
A bilinear map $\phi$ of $\SV(\varepsilon)$ is a biderivation if and only if there are
$\lambda\in\C$ and complex numbers $\{\mu_i\}_{i\in\Z}$, only finitely many of which are nonzero, such that
$\phi(x,y)=\lambda[x,y]+\chi(x,y)$ for all $x,y\in\SV(\varepsilon)$, where $\chi$ is the
bilinear map defined by
\[
\chi(L_m,L_n)=\sum\limits_{i\in\Z}\mu_iM_{m+n+i}\quad (m,n\in\Z),
\]
and $\chi(x,y)=0$ if one of the basis elements $x,y$ belongs to
$\{Y_n,M_m\mid m\in\Z,\ n\in\varepsilon+\Z\}$.
\end{theorem}

\begin{corollary}\label{cor:bider}
Let $\varepsilon\in\{0,\frac12\}$ and $\delta\in\C$.
\begin{enumerate}[{\rm (1)}]
\item If $\delta\ne1$, then every $\delta$-biderivation of $\SV(\varepsilon)$ is zero.
\item If $\delta=1$, then the $\delta$-biderivations of $\SV(\varepsilon)$ are the
biderivations described in Theorem \ref{thm:LTbider}.
\end{enumerate}
\end{corollary}

\begin{proof}
The second statement is Theorem \ref{thm:LTbider}. Let $\delta\ne1$ and let $\phi$ be a
$\delta$-biderivation. For every $z$ the maps $x\mapsto\phi(x,z)$ and $x\mapsto\phi(z,x)$ are
$\delta$-derivations. If $\delta\neq\frac12$, they vanish by Corollary \ref{cor:delta}, so
$\phi=0$. Let $\delta=\frac12$. By Corollary \ref{cor:delta}, there are maps
$\lambda,\mu:\SV(\varepsilon)\to\C$ such that $\phi(x,y)=\lambda(y)x=\mu(x)y$ for all $x,y$.
Given $y$, choose a basis element $P$ which does not occur in the decomposition of $y$ with
respect to the basis. Comparing the coefficients at $P$ in $\lambda(y)P=\phi(P,y)=\mu(P)y$, we
get $\lambda(y)=0$. Hence $\phi(x,y)=\lambda(y)x=0$ for all $x,y$.
\end{proof}

Recall that transposed $\delta$-Poisson structures were defined in Definition \ref{def:tp};
for $\delta=\frac12$ these are exactly the transposed Poisson structures. By Lemma
\ref{lem:tpbider}, if $\cdot$ is a transposed $\delta$-Poisson structure on a Lie algebra $L$,
then the bilinear map $\phi(x,y)=x\cdot y$ is a $\delta$-biderivation of $L$.

\begin{corollary}\label{cor:tp}
If $\delta\neq1$, then every transposed $\delta$-Poisson structure on $\SV(\varepsilon)$ is trivial.
If $\delta=1$, then the transposed $1$-Poisson structures on $\SV(\varepsilon)$ are exactly
the multiplications given by
\[
L_m\cdot L_n=\sum_{i\in\Z}\mu_iM_{m+n+i}\qquad(m,n\in\Z),
\]
with all other products of basis elements equal to zero, where $\mu_i\in\C$ and only
finitely many of the $\mu_i$ are nonzero.
\end{corollary}

\begin{proof}
Let $\cdot$ be a transposed $\delta$-Poisson structure. By Lemma \ref{lem:tpbider},
$\phi(x,y)=x\cdot y$ is a $\delta$-biderivation, so $\phi=0$ if $\delta\neq1$ by Corollary
\ref{cor:bider}. Let $\delta=1$. By Theorem \ref{thm:LTbider}, $x\cdot y=\lambda[x,y]+\chi(x,y)$.
Since both $\cdot$ and $\chi$ are commutative, $\lambda[x,y]=\lambda[y,x]$ for all $x,y$; for
$x=L_1$, $y=L_0$ this gives $2\lambda L_1=0$, so $\lambda=0$ and $x\cdot y=\chi(x,y)$.
Conversely, let $x\cdot y=\chi(x,y)$. This multiplication is commutative. It is associative,
since all products lie in $\Span{M_m\mid m\in\Z}$ and $M_k\cdot y=0$ for all $k$ and $y$, so
that $(x\cdot y)\cdot z=0=x\cdot(y\cdot z)$. Finally, \eqref{eq:tdp} with $\delta=1$ is the identity
\eqref{eq:bd2} for the biderivation $\chi$ (Theorem \ref{thm:LTbider}), with $(x,y,z)$ replaced by $(z,x,y)$.
\end{proof}

\section{The extended Schr\"odinger--Witt algebras}\label{sec:esv}

\begin{definition}\label{def:esv}
Let $\varepsilon\in\{0,\frac12\}$. The \emph{extended Schr\"odinger--Witt algebra}
$\ESV(\varepsilon)$ is the Lie algebra with the basis
$\{L_m,\ M_m,\ N_m,\ Y_n\mid m\in\Z,\ n\in\varepsilon+\Z\}$ and the following nontrivial
relations:
\begin{equation}\label{eq:esv}
\begin{aligned}
&[L_m,L_n]=(m-n)L_{m+n}, && [L_m,Y_n]=\Big(\frac m2-n\Big)Y_{m+n},
&& [L_m,M_n]=-nM_{m+n},\\
&[Y_m,Y_n]=(m-n)M_{m+n}, && [L_m,N_n]=-nN_{m+n},
&& [N_m,Y_n]=Y_{m+n},\\
&[N_m,M_n]=2M_{m+n}. && &&
\end{aligned}
\end{equation}
All other brackets of basis elements are zero (in particular, $[N_m,N_n]=0$).
\end{definition}

The algebra $\ESV(\frac12)$ is the extended Schr\"odinger--Witt algebra (written in the
conventions of the present paper; see \cite{YH22} and the references therein), and
$\ESV(0)$ is its twisted version. A direct verification on triples of basis elements shows
that the Jacobi identity holds. Clearly, $\SV(\varepsilon)=\Span{L_m,M_m,Y_n}$ is a
subalgebra of $\ESV(\varepsilon)$. As for $\SV(\varepsilon)$, the indices of $L_m$,
$M_m$, $N_m$ belong to $\Z$ and the indices of $Y_n$ belong to $\varepsilon+\Z$.

The algebra $\ESV(\varepsilon)$ is graded by $G=\frac12\Z\times\Z_2$:
\[
\ESV(\varepsilon)_{(t,\bar0)}=\Span{L_t,M_t,N_t}\ (t\in\Z),\qquad
\ESV(\varepsilon)_{(t,\bar1)}=\Span{Y_t}\ (t\in\varepsilon+\Z),
\]
and all other homogeneous components are zero. As for $\SV(\varepsilon)$, we speak
about \emph{even} linear maps of degree $j\in\Z$ and \emph{odd} linear maps of degree
$s\in\varepsilon+\Z$.

  Note that the restriction of $\ad_{N_0}$ to the subalgebra $\SV(\varepsilon)$ of
$\ESV(\varepsilon)$ is the outer derivation $D_3$ of Theorem \ref{thm:der}. For every
$t\in\Z$ and $s\in\varepsilon+\Z$ the inner derivations act as follows (we list only the
values not given by \eqref{eq:ad}):
\begin{equation}\label{eq:adesv}
\begin{aligned}
&\ad_{L_t}(N_m)=-mN_{m+t}, && \ad_{M_t}(N_m)=-2M_{m+t}, && \ad_{Y_s}(N_m)=-Y_{m+s},\\
&\ad_{N_t}(L_m)=tN_{m+t}, && \ad_{N_t}(Y_n)=Y_{n+t}, && \ad_{N_t}(M_m)=2M_{m+t}.
\end{aligned}
\end{equation}

Note also that $\ad_{N_0}$ acts diagonally: $$[N_0,L_m]=[N_0,N_m]=0, \quad [N_0,Y_n]=Y_n, \quad [N_0,M_m]=2M_m.$$

\begin{lemma}[see \cite{GJP09} for $\varepsilon=\frac12$]\label{lem:esvperfect}
The algebra $\ESV(\varepsilon)$ is perfect and $Z(\ESV(\varepsilon))=0$.
\end{lemma}

\begin{proof}
We have 
\begin{center}$L_m=\frac1m[L_m,L_0]$ for $m\neq0$, $L_0=\frac12[L_1,L_{-1}]$,\\
$N_m=-[L_{m-1},N_1]$, $M_m=\frac12[N_0,M_m]$ and $Y_n=[N_0,Y_n]$.\end{center} Hence
$\ESV(\varepsilon)$ is perfect. Let $z$ be a central element. Since
\begin{center}
$[L_0,L_m]=-mL_m$, $[L_0,Y_n]=-nY_n$, $[L_0,M_n]=-nM_n$ and $[L_0,N_n]=-nN_n$,
\end{center} the equality
$[L_0,z]=0$ gives $z=\lambda L_0+\mu Y_0+\nu M_0+\rho N_0$ (with $\mu=0$ if
$\varepsilon=\frac12$). Then $0=[N_1,z]=\lambda N_1+\mu Y_1+2\nu M_1$ gives
$\lambda=\mu=\nu=0$, and $0=[z,Y_n]=\rho Y_n$ gives $\rho=0$.
\end{proof}

By Lemmas \ref{lem:unique} and \ref{lem:esvperfect}, the map $D'$ associated with a
quasi-derivation $D$ of $\ESV(\varepsilon)$ is unique, and $\ESV(\varepsilon)$ has no nonzero
central maps. Since every homogeneous component of a quasi-derivation of
$\ESV(\varepsilon)$ is a quasi-derivation (Lemma \ref{lem:grading}), the associated map of a homogeneous
quasi-derivation is homogeneous of the same degree.

\begin{remark}\label{rem:esvknown}
The structure of the original extended Schr\"odinger--Witt algebra $\ESV(\frac12)$ was
studied by Gao, Jiang and Pei \cite{GJP09}; in particular, they proved that every derivation
of $\ESV(\frac12)$ is inner. This result is recovered, for both algebras $\ESV(\frac12)$ and
$\ESV(0)$, in Corollary \ref{cor:esvder} below.
The $\frac12$-derivations and the transposed Poisson structures of the original extended
Schr\"odinger--Witt algebra $\ESV(\frac12)$ were studied in \cite{YH22}:
every $\frac12$-derivation of $\ESV(\frac12)$ is trivial and all transposed Poisson structures
on $\ESV(\frac12)$ are trivial. These results are recovered, for both algebras $\ESV(\frac12)$
and $\ESV(0)$, in Corollaries \ref{cor:esvdelta} and \ref{cor:esvtp} below.
\end{remark}

\subsection{Quasi-derivations of the extended Schr\"odinger--Witt algebras}

\begin{proposition}\label{prop:esveven}
Let $D$ be an even quasi-derivation of degree $j\in\Z$ of $\ESV(\varepsilon)$. Then there
are $u,v,w,\lambda\in\C$, with $\lambda=0$ if $j\neq0$, such that
\[
D=u\ad_{L_j}+v\ad_{M_j}+w\ad_{N_j}+\lambda\Id.
\]
\end{proposition}

\begin{proof}
Since $D$ and $D'$ are even of degree $j$, we can write
\[
\begin{aligned}
&D(L_m)=a_mL_{m+j}+b_mM_{m+j}+r_mN_{m+j}, && D(Y_n)=c_nY_{n+j},\\
&D(M_m)=d_mL_{m+j}+e_mM_{m+j}+p_mN_{m+j}, && D(N_m)=x_mL_{m+j}+y_mM_{m+j}+z_mN_{m+j},
\end{aligned}
\]
and similarly for $D'$, with the coefficients denoted by $a'_m,b'_m,\dots,z'_m$. Applying
\eqref{eq:qd} to the pairs of basis elements indicated in brackets and comparing the
coefficients at the basis elements indicated in brackets, we obtain the following relations:
\begin{align}
&\Big(m-\frac{n+j}2\Big)d_n-p_n=0 && [(Y_m,M_n),\ Y_{m+n+j}], \label{x1}\\
& y_m=y_n && [(N_m,N_n),\ M_{m+n+j}], \label{x2_1}\\
&-nx_m+mx_n=0 && [(N_m,N_n),\ N_{m+n+j}], \label{x2_2}\\
&(n+j-m)x_n=nx'_{m+n} && [(L_m,N_n),\ L_{m+n+j}], \label{x3}\\
&2b_m+(n+j)y_n=ny'_{m+n} && [(L_m,N_n),\ M_{m+n+j}], \label{x4}\\
&na_m+(n+j)z_n=nz'_{m+n} && [(L_m,N_n),\ N_{m+n+j}], \label{x5}\\
&c_m+\Big(\frac{n+j}2-m\Big)x_n+z_n=c'_{m+n} && [(Y_m,N_n),\ Y_{m+n+j}], \label{x6}\\
&e_m-\frac m2x_n+z_n=e'_{m+n} && [(M_m,N_n),\ M_{m+n+j}], \label{x7}\\
&-na_m+2r_m-(n+j)e_n=-ne'_{m+n} && [(L_m,M_n),\ M_{m+n+j}], \label{x8}\\
&(m-n+j)c_m+(m-n-j)c_n=(m-n)e'_{m+n} && [(Y_m,Y_n),\ M_{m+n+j}], \label{x9}
\end{align}
where $m,n\in\Z$ if the corresponding basis elements are $L$, $M$ or $N$, and
$m,n\in\varepsilon+\Z$ if they are $Y$. For example, for the pair $(L_m,N_n)$ we have
\[
[D(L_m),N_n]+[L_m,D(N_n)]=\]\[(m-n-j)x_nL_{m+n+j}-\big(2b_m+(n+j)y_n\big)M_{m+n+j}
-\big(na_m+(n+j)z_n\big)N_{m+n+j}
\]
and $D'([L_m,N_n])=-n\big(x'_{m+n}L_{m+n+j}+y'_{m+n}M_{m+n+j}+z'_{m+n}N_{m+n+j}\big)$, which
gives \eqref{x3}, \eqref{x4} and \eqref{x5}; the other relations are obtained in the same way.

\smallskip
\begin{enumerate}
    \item[\bf{Step 1.}] Taking two different values of $m\in\varepsilon+\Z$ in \eqref{x1}, we get
$d_n=0$ and then $p_n=0$ for all $n\in\Z$.

\smallskip
    \item[\bf{Step 2.}] By \eqref{x2_2} with $n=1$, $x_m=mx_1$ for all $m$. By
\eqref{x3} with $n=1$ we have $(1+j-m)x_1=x'_{m+1}$, and by \eqref{x3} with $n=2$ and $m$
replaced by $m-1$ we have $(3+j-m)x_2=2x'_{m+1}$. Since $x_2=2x_1$, we get
$(3+j-m)x_1=(1+j-m)x_1$, i.e., $x_1=0$. Hence $x_m=0$ for all $m\in\Z$.

\smallskip
    \item[\bf{Step 3.}] By \eqref{x2_1}, there is $y\in\C$ with $y_m=y$ for all
$m$. Then \eqref{x4} with $n=0$ gives $b_m=-\frac j2y$ for all $m$.

\smallskip
    \item[\bf{Step 4.}] Since $x_n=0$, the relation \eqref{x7} reads $e_m+z_n=e'_{m+n}$. Comparing
this relation for the pairs $(m+1,n)$ and $(m,n+1)$, we get $e_{m+1}-e_m=z_{n+1}-z_n$ for all
$m,n\in\Z$. Hence this difference is a constant, which we denote by $-u$, and there are
$\lambda,\eta\in\C$ such that
\[
z_n=\lambda-un,\qquad e_m=\eta-um,\qquad e'_k=e_0+z_k=\eta+\lambda-uk\quad(m,n,k\in\Z).
\]
In the same way, \eqref{x6} reads $c_m+z_n=c'_{m+n}$, whence $c_{m+1}-c_m=z_{n+1}-z_n=-u$, and
there is $\gamma\in\C$ with $c_m=\gamma-um$ for all $m\in\varepsilon+\Z$.

\smallskip
    \item[\bf{Step 5.}] The relation \eqref{x5} with $n=0$ gives $jz_0=0$, i.e.,
\begin{equation}\label{x10}
j\lambda=0.
\end{equation}
By \eqref{x5} with $n=1$ (and $m$ replaced by $m+1$) and with $n=2$, we have\begin{center}
$a_{m+1}+(1+j)z_1=z'_{m+2}$ and $2a_m+(2+j)z_2=2z'_{m+2}$. \end{center}Hence
$2(a_{m+1}-a_m)=(2+j)z_2-2(1+j)z_1=-j\lambda-2u=-2u$, and $a_m=a_0-um$ for all $m\in\Z$.

\smallskip
    \item[\bf{Step 6.}] By \eqref{x8} with $n=0$ we get $r_m=\frac j2\eta$ for all $m$. By \eqref{x8}
with $m=0$ and $n=1$ we get $-a_0+j\eta-(1+j)(\eta-u)=-(\eta+\lambda-u)$, i.e.,
$a_0=uj+\lambda$. Thus $a_m=u(j-m)+\lambda$.

\smallskip
    \item[\bf{Step 7.}] Substituting $c_m=\gamma-um$ and $e'_k=\eta+\lambda-uk$ in \eqref{x9}, we get
$(m-n)(2\gamma-u(m+n)-uj)=(m-n)(\eta+\lambda-u(m+n))$, and taking $m\ne n$ we obtain
$\gamma=\frac12(\eta+\lambda+uj)$.

\smallskip
    \item[\bf{Step 8.}] Let $v=-\frac y2$ and $w=\frac12(\eta-\lambda)$. By \eqref{eq:ad},
\eqref{eq:adesv} and \eqref{x10}, the map $\widehat D=u\ad_{L_j}+v\ad_{M_j}+w\ad_{N_j}+\lambda\Id$
satisfies
\[
\begin{aligned}
\widehat D(L_m)&=\big(u(j-m)+\lambda\big)L_{m+j}+vjM_{m+j}+wjN_{m+j}
=a_mL_{m+j}+b_mM_{m+j}+r_mN_{m+j},\\
\widehat D(Y_n)&=\Big(u\Big(\frac j2-n\Big)+w+\lambda\Big)Y_{n+j}=c_nY_{n+j},\\
\widehat D(M_m)&=(-um+2w+\lambda)M_{m+j}=e_mM_{m+j},\\
\widehat D(N_m)&=-2vM_{m+j}+(\lambda-um)N_{m+j}=y_mM_{m+j}+z_mN_{m+j},
\end{aligned}
\]
where we used $jw=\frac j2\eta$, which holds by \eqref{x10}. By Steps 1--7, $D=\widehat D$.
Finally, $\lambda=0$ if $j\ne0$ by \eqref{x10}.

\end{enumerate}\end{proof}

\begin{proposition}\label{prop:esvodd}
Let $D$ be an odd quasi-derivation of degree $s\in\varepsilon+\Z$ of $\ESV(\varepsilon)$.
Then there is $w\in\C$ such that $D=w\ad_{Y_s}$.
\end{proposition}

\begin{proof}
Since $D$ and $D'$ are odd of degree $s$, we can write
\[
D(L_m)=f_mY_{m+s},\quad D(M_m)=g_mY_{m+s},\quad D(N_m)=q_mY_{m+s},\]\[
D(Y_n)=h_nL_{n+s}+k_nM_{n+s}+l_nN_{n+s},
\]
and similarly for $D'$. As in the proof of Proposition \ref{prop:esveven}, we obtain the
following relations:
\begin{align}
&g_m=2g'_{m+n} && [(M_m,N_n),\ Y_{m+n+s}], \label{y1}\\
&\Big(\frac m2-n-s\Big)g_n=-ng'_{m+n} && [(L_m,M_n),\ Y_{m+n+s}], \label{y2}\\
&h'_{m+n}=0 && [(Y_m,N_n),\ L_{m+n+s}],
\label{y3_1}\\
& 2k_m+(n+s-m)q_n=k'_{m+n} && [(Y_m,N_n),\ M_{m+n+s}],
\label{y3_2}\\
&(m-n-s)h_n=\Big(\frac m2-n\Big)h'_{m+n} && [(L_m,Y_n),\ L_{m+n+s}], \label{y4}\\
&(m+s-n)f_m-(n+s)k_n=\Big(\frac m2-n\Big)k'_{m+n} && [(L_m,Y_n),\ M_{m+n+s}], \label{y5}\\
&-nh_m+2l_m+(m-n-s)g_n=0 && [(Y_m,M_n),\ M_{m+n+s}], \label{y6}\\
&q_m=q_n && [(N_m,N_n),\ Y_{m+n+s}], \label{y7}\\
&-f_m+\Big(\frac m2-n-s\Big)q_n=-nq'_{m+n} && [(L_m,N_n),\ Y_{m+n+s}]. \label{y8}
\end{align}

\smallskip
\begin{enumerate}
    \item[\bf{Step 1.}] By \eqref{y1} for the pairs $(m,n)$ and $(m+n,0)$, we get
$g_m=2g'_{m+n}=g_{m+n}$, so $g_m=g_0$ for all $m$. Now \eqref{y2} with $n=0$ gives
$(\frac m2-s)g_0=0$ for all $m\in\Z$, hence $g_m=0$ for all $m\in\Z$.

\smallskip
\item[\bf{Step 2.}] By \eqref{y3_1}, $h'=0$; then \eqref{y4} gives $(m-n-s)h_n=0$ for all $m\in\Z$,
hence $h_n=0$ for all $n$. Now \eqref{y6} gives $l_m=0$ for all $m$.

\smallskip
\item[\bf{Step 3.}] By \eqref{y7}, there is $q\in\C$ with $q_m=q$ for all $m$, and \eqref{y8} with
$n=0$ gives $f_m=(\frac m2-s)q$.

\smallskip
\item[\bf{Step 4.}] Comparing the relation \eqref{y3_2} for the pairs $(m+1,n)$ and
$(m,n+1)$, we get $k_{m+1}-k_m=q$, so there is $\rho\in\C$ with $k_m=\rho+qm$ for all
$m\in\varepsilon+\Z$. Now take $m=2n\in\Z$ in \eqref{y5}: we get
$(n+s)(f_{2n}-k_n)=0$, where $f_{2n}-k_n=(n-s)q-\rho-qn=-(\rho+qs)$. Choosing
$n\ne-s$, we obtain $\rho=-qs$, i.e., $k_n=q(n-s)$.

\smallskip
\item[\bf{Step 5.}] Let $w=-q$. By \eqref{eq:ad} and \eqref{eq:adesv},
\begin{center}$w\ad_{Y_s}(L_m)=-q(s-\frac m2)Y_{m+s}=f_mY_{m+s}$,
$w\ad_{Y_s}(Y_n)=-q(s-n)M_{n+s}=k_nM_{n+s}$, $w\ad_{Y_s}(N_m)=qY_{m+s}=q_mY_{m+s}$ and
$w\ad_{Y_s}(M_m)=0$.\end{center} By Steps 1--4, $D=w\ad_{Y_s}$.
\end{enumerate}
\end{proof}

\begin{theorem}\label{thm:esvmain}
A linear map $D$ of $\ESV(\varepsilon)$ is a quasi-derivation if and only if there are
$x\in\ESV(\varepsilon)$ and $\lambda\in\C$ such that
$D=\ad_x+\lambda\Id.$
The element $x$ and the number $\lambda$ are uniquely determined by $D$, and the map
associated with $D$ is $D'=\ad_x+2\lambda\Id$. In other words,
\[
\QDer\big(\ESV(\varepsilon)\big)=\Inn\big(\ESV(\varepsilon)\big)\oplus\Span{\Id}.
\]
\end{theorem}

\begin{proof}
Clearly, $\ad_x+\lambda\Id$ is a quasi-derivation with associated map
$\ad_x+2\lambda\Id$. Conversely, let $D$ be a quasi-derivation. By Propositions
\ref{prop:esveven} and \ref{prop:esvodd}, there are complex numbers $u_j,v_j,w_j$
($j\in\Z$), $\lambda$ and $w'_s$ ($s\in\varepsilon+\Z$) such that the even component of degree
$j$ of $D$ equals \begin{center}$u_j\ad_{L_j}+v_j\ad_{M_j}+w_j\ad_{N_j}$ for $j\ne0$ and
$u_0\ad_{L_0}+v_0\ad_{M_0}+w_0\ad_{N_0}+\lambda\Id$ for $j=0$,\end{center} and the odd component of degree
$s$ of $D$ equals $w'_s\ad_{Y_s}$. By \eqref{eq:ad} and \eqref{eq:adesv}, for $j\ne0$ the
coefficient of $D(N_1)$ at $N_{1+j}$ equals $-u_j$, the coefficient of $D(N_0)$ at $M_j$
equals $-2v_j$ and the coefficient of $D(M_0)$ at $M_j$ equals $2w_j$; moreover, the
coefficient of $D(N_0)$ at $Y_s$ equals $-w'_s$. Since $D(N_0)$, $D(N_1)$ and $D(M_0)$ are
finite linear combinations of basis elements, only finitely many of the numbers
$u_j,v_j,w_j,w'_s$ are nonzero. Hence
$x=\sum\limits_{j\in\Z}(u_jL_j+v_jM_j+w_jN_j)+\sum\limits_{s\in\varepsilon+\Z}w'_sY_s$ is an element of
$\ESV(\varepsilon)$, and $D$ and $\ad_x+\lambda\Id$ have the same homogeneous components, so
$D=\ad_x+\lambda\Id$.

Let us prove the uniqueness. Assume that $\ad_x+\lambda\Id=0$. Let $\xi$ be the coefficient
of $x$ at $L_0$. Comparing the coefficients at $N_1$ in $[x,N_1]+\lambda N_1=0$ and at $N_2$ in
$[x,N_2]+\lambda N_2=0$, we get $-\xi+\lambda=0$ and $-2\xi+\lambda=0$, so $\lambda=0$. Then
$\ad_x=0$, and $x=0$ by Lemma \ref{lem:esvperfect}. The formula for $D'$ follows from the
uniqueness of associated maps.
\end{proof}

\subsection{Corollaries}

\begin{corollary}\label{cor:esvder}
Let $\varepsilon\in\{0,\frac12\}$. Then
\begin{enumerate}[{\rm (1)}]
\item $\Der\big(\ESV(\varepsilon)\big)=\Inn\big(\ESV(\varepsilon)\big)$, i.e., all derivations of
$\ESV(\varepsilon)$ are inner (for the original extended Schr\"odinger--Witt algebra
$\ESV(\frac12)$ see \cite{GJP09}).
\item $\QDer\big(\ESV(\varepsilon)\big)=\Der\big(\ESV(\varepsilon)\big)\oplus\Delta_{\frac{1}{2}}\big(\ESV(\varepsilon)\big)$.
\end{enumerate}
\end{corollary}

\begin{proof}
By Theorem \ref{thm:esvmain} and Lemma \ref{lem:unique}, a quasi-derivation
$D=\ad_x+\lambda\Id$ is a derivation if and only if $\ad_x+\lambda\Id=D=D'=\ad_x+2\lambda\Id$,
i.e., if and only if $\lambda=0$. This proves the first statement. The second statement
follows from the first one, Theorem \ref{thm:esvmain} and Corollary \ref{cor:esvdelta} below
(for $\delta=\frac12$).
\end{proof}

\begin{corollary}\label{cor:esvdelta}
Let $\varepsilon\in\{0,\frac12\}$ and $\delta\in\C$. Then
\[
\Delta_\delta\big(\ESV(\varepsilon)\big)=
\begin{cases}
\Inn\big(\ESV(\varepsilon)\big), & \delta=1,\\
\Span{\Id}, & \delta=\frac12,\\
0, & \delta\notin \{\frac12,1\}.
\end{cases}
\]
In particular, every $\frac12$-derivation of $\ESV(\varepsilon)$ is trivial (for the original
extended Schr\"odinger--Witt algebra $\ESV(\frac12)$ cf. \cite{YH22}).
\end{corollary}

\begin{proof}
If $\delta=0$, then a $0$-derivation vanishes on
$[\ESV(\varepsilon),\ESV(\varepsilon)]=\ESV(\varepsilon)$ (Lemma \ref{lem:esvperfect}). Let
$\delta\neq0$ and put $c=\delta^{-1}$. A $\delta$-derivation $\varphi$ is a quasi-derivation
with $\varphi'=c\varphi$. By Theorem \ref{thm:esvmain}, $\varphi=\ad_x+\lambda\Id$ and
$\ad_x+2\lambda\Id=\varphi'=c\ad_x+c\lambda\Id$, i.e.,
$\ad_{(1-c)x}+(2-c)\lambda\Id=0$. By the uniqueness part of Theorem \ref{thm:esvmain},
$(1-c)x=0$ and $(2-c)\lambda=0$. If $\delta=1$, then $\lambda=0$ and $\varphi=\ad_x$; if
$\delta=\frac12$, then $x=0$ and $\varphi=\lambda\Id$; otherwise $x=0$ and $\lambda=0$.
Conversely, inner derivations are $1$-derivations and $\Id$ is a $\frac12$-derivation.
\end{proof}

To describe the biderivations of $\ESV(\varepsilon)$ we need the following lemma on its
quasicentroid (in the sense of \cite{LL00}).

\begin{lemma}\label{lem:qc}
Let $T$ be a linear map of $\ESV(\varepsilon)$ such that $[T(x),y]=[x,T(y)]$ for all
$x,y\in\ESV(\varepsilon)$. Then $T=\beta\Id$ for some $\beta\in\C$.
\end{lemma}

\begin{proof}
The identity $[T(x),y]=[x,T(y)]$ is homogeneous, so, as in Lemma \ref{lem:grading}, every
homogeneous component of $T$ satisfies it. Let $T$ be even of degree $j$, and write
\begin{center}$T(L_m)=a_mL_{m+j}+b_mM_{m+j}+r_mN_{m+j}$, $T(Y_n)=c_nY_{n+j}$,\\
$T(M_m)=d_mL_{m+j}+e_mM_{m+j}+p_mN_{m+j}$, $T(N_m)=x_mL_{m+j}+y_mM_{m+j}+z_mN_{m+j}$.\end{center}
Comparing the coefficients at the basis elements indicated in brackets in
$[T(X),Z]=[X,T(Z)]$ for the indicated pairs $(X,Z)$ of basis elements, we get
\[
{\small\begin{aligned}
&nx_m+mx_n=0 &&[(N_m,N_n),\ N_{m+n+j}], && y_m=0 &&[(N_m,N_m),\ M_{2m+j}],\\
&na_m=(n+j)z_n &&[(L_m,N_n),\ N_{m+n+j}], && 2b_m=jy_0 &&[(L_m,N_0),\ M_{m+j}],\\
&2r_m=na_m-(n+j)e_n &&[(L_m,M_n),\ M_{m+n+j}], && e_m-z_n=-\tfrac m2x_n &&[(M_m,N_n),\ M_{m+n+j}],\\
&c_m-z_n=\Big(\frac{n+j}2-m\Big)x_n &&[(Y_m,N_n),\ Y_{m+n+j}], && \Big(\frac{m+j}2-n\Big)d_m+p_m=0 &&[(M_m,Y_n),\ Y_{m+n+j}].
\end{aligned}}
\]
The first relation with $m=n=1$ and then with $n=1$ gives $x_m=0$ for all $m$. Hence
$e_m=z_n=c_k$ for all $m,n\in\Z$, $k\in\varepsilon+\Z$; we denote this number by $\beta_j$.
Moreover, $b_m=0$; the third relation with $n=0$ and $n=1$ gives $j\beta_j=0$ and
$a_m=(1+j)\beta_j=\beta_j$; the fifth relation with $n=0$ gives $2r_m=-j\beta_j=0$; and the last
relation for two different values of $n$ gives $d_m=p_m=0$. Thus $T=\beta_j\Id$, and
$\beta_j=0$ if $j\neq0$.
Let $T$ be odd of degree $s$, and write \begin{center}$T(L_m)=f_mY_{m+s}$, $T(M_m)=g_mY_{m+s}$,
$T(N_m)=q_mY_{m+s}$, $T(Y_n)=h_nL_{n+s}+k_nM_{n+s}+l_nN_{n+s}$.\end{center} In the same way, the pairs
\begin{center}$(N_m,N_m)$, $(L_m,N_0)$, $(M_m,N_0)$, $(Y_n,N_0)$, $(Y_n,N_1)$ and $(Y_n,M_0)$,\end{center} compared at
$Y_{2m+s}$, $Y_{m+s}$, $Y_{m+s}$, $M_{n+s}$, $N_{n+s+1}$ and $M_{n+s}$, respectively, give
successively $q_m=0$, $f_m=0$, $g_m=0$, $k_n=0$, $h_n=0$ and $l_n=0$. Hence $T=0$. Therefore
$T$ is equal to its even component of degree $0$, which is a multiple of $\Id$.
\end{proof}

\begin{corollary}\label{thm:esvbider}
Let $\varepsilon\in\{0,\frac12\}$ and $\delta\in\C$.
\begin{enumerate}[{\rm (1)}]
\item If $\delta\ne1$, then every $\delta$-biderivation of $\ESV(\varepsilon)$ is zero.
\item A bilinear map $\phi$ of $\ESV(\varepsilon)$ is a biderivation if and only if there is
$\beta\in\C$ such that $\phi(x,y)=\beta[x,y]$ for all $x,y\in\ESV(\varepsilon)$.
\end{enumerate}
\end{corollary}

\begin{proof}
The first statement is proved in the same way as the first part of Corollary \ref{cor:bider},
using Corollary \ref{cor:esvdelta} instead of Corollary \ref{cor:delta}. Let $\phi$ be a
biderivation. By Corollary \ref{cor:esvder}, for every $z$ there are $a(z),b(z)\in\ESV(\varepsilon)$
such that $\phi(x,z)=[a(z),x]$ and $\phi(z,x)=[b(z),x]$ for all $x$. Since
$Z(\ESV(\varepsilon))=0$, the elements $a(z),b(z)$ are uniquely determined, and the maps $a$ and
$b$ are linear. For all $x,y$ we have
\[
[(a+b)(x),y]+[x,(a+b)(y)]=\phi(y,x)+\phi(x,y)-\phi(x,y)-\phi(y,x)=0,
\]
so $a+b$ is a quasi-derivation with associated map $0$. By Theorem \ref{thm:esvmain},
$a+b=\ad_w+\mu\Id$ with $\ad_w+2\mu\Id=0$, whence $w=0$, $\mu=0$ and $a=-b$. Therefore
$[b(x),y]=\phi(x,y)=[a(y),x]=[x,b(y)]$ for all $x,y$, and $b=\beta\Id$ by Lemma \ref{lem:qc}.
Hence $\phi(x,y)=\beta[x,y]$. Conversely, $\beta[x,y]$ is a biderivation by the Jacobi identity.
\end{proof}

\begin{corollary}\label{cor:esvtp}
Let $\varepsilon\in\{0,\frac12\}$. For every $\delta\in\C$, all transposed $\delta$-Poisson
structures on $\ESV(\varepsilon)$ are trivial. In particular, all transposed Poisson structures
on $\ESV(\varepsilon)$ are trivial (for the original extended Schr\"odinger--Witt algebra
$\ESV(\frac12)$ cf. \cite{YH22}).
\end{corollary}

\begin{proof}
By Lemma \ref{lem:tpbider}, $\phi(x,y)=x\cdot y$ is a $\delta$-biderivation. If $\delta\neq1$, then
$\phi=0$ by Corollary \ref{thm:esvbider}. If $\delta=1$, then $x\cdot y=\beta[x,y]$, and the
commutativity of $\cdot$ gives $\beta L_1=\beta[L_1,L_0]=\beta[L_0,L_1]=-\beta L_1$, so $\beta=0$.
\end{proof}

\begin{remark}
The identity \eqref{eq:tdp} is sometimes normalized as
$\delta\, z\cdot[x,y]=[z\cdot x,y]+[x,z\cdot y]$. For $\delta\neq0$ this is \eqref{eq:tdp} with
$\delta$ replaced by $\delta^{-1}$, and for $\delta=0$ it means that every multiplication
operator $D=z\cdot(-)$ is a quasi-derivation with $D'=0$; by Theorem \ref{thm:esvmain},
$D=\ad_x+\lambda\Id$ with $0=D'=\ad_x+2\lambda\Id$, whence $x=0$ and $\lambda=0$ by the
uniqueness part of Theorem \ref{thm:esvmain}. Hence Corollary \ref{cor:esvtp} also holds in this
normalization, for every $\delta\in\C$.
\end{remark}

\end{document}